\documentclass{amsart}

\usepackage{amssymb,amsmath,amsthm,latexsym,booktabs}

\theoremstyle{definition}
\newtheorem{definition}{Definition}[section]

\newtheorem*{openproblem}{Open Problem}

\theoremstyle{plain}
\newtheorem{lemma}[definition]{Lemma}
\newtheorem{theorem}[definition]{Theorem}

\begin{document}


\title[Fundamental invariants for $3 \times 3 \times 3$ arrays]
{Fundamental invariants for the action of
$SL_3(\mathbb{C}) \times SL_3(\mathbb{C}) \times SL_3(\mathbb{C})$
on $3 \times 3 \times 3$ arrays}

\author{Murray R. Bremner}

\address{Department of Mathematics and Statistics, University of Saskatchewan, Canada}

\email{bremner@math.usask.ca}

\author{Jiaxiong Hu}

\address{Department of Mathematics, Simon Fraser University, Canada}

\email{hujiaxiong@gmail.com}

\date{28 March 2012}

\begin{abstract}
We determine the three fundamental invariants in the entries of a $3 \times 3 \times 3$ array
over $\mathbb{C}$ as explicit polynomials in the 27 variables $x_{ijk}$ for $1 \le i, j, k \le 3$.
By the work of Vinberg on $\theta$-groups, it is known that these homogeneous polynomials have degrees 6, 9 and 12;
they freely generate the algebra of invariants for the Lie group
$SL_3(\mathbb{C}) \times SL_3(\mathbb{C}) \times SL_3(\mathbb{C})$
acting irreducibly on its natural representation $\mathbb{C}^3 \otimes \mathbb{C}^3 \otimes \mathbb{C}^3$.
These generators have respectively 1152, 9216 and 209061 terms;
we find compact expressions in terms of the orbits of the finite group
$( S_3 \times S_3 \times S_3 ) \rtimes S_3$
acting on monomials of weight zero for the action of the Lie algebra
$\mathfrak{sl}_3(\mathbb{C}) \oplus \mathfrak{sl}_3(\mathbb{C}) \oplus \mathfrak{sl}_3(\mathbb{C})$.
\end{abstract}



\maketitle


\section{Introduction}

We consider $3 \times 3 \times 3$ arrays $X = ( x_{ijk} )$ where $x_{ijk} \in \mathbb{C}$ for $1 \le i, j, k \le 3$.
We are concerned with homogeneous polynomials in the entries $x_{ijk}$ of these arrays which are invariant under the action
of the Lie group $SL_3(\mathbb{C}) \times SL_3(\mathbb{C}) \times SL_3(\mathbb{C})$ acting by simultaneous changes
of basis along the three directions.

Following Bremner et al.~\cite{Bremner,BBS}, we use a computational approach, based on the representation theory
of the special linear Lie algebra $\mathfrak{sl}_3(\mathbb{C})$, to express these invariant polynomials as elements
in the nullspace of a large integer matrix, and then to determine explicitly the generators of the algebra of invariants.
Apart from their intrinsic interest, these generators can in principle be used to determine an explicit form for
the hyperdeterminant of a $3 \times 3 \times 3$ array in the sense of Gelfand et al. \cite{GKZ}; see the open problem
at the end of this paper.

Vinberg \cite{Vinberg} generalized the notion of Weyl group from Lie groups to $\theta$-groups; a $\theta$-group
is the group of fixed points of an automorphism of finite order of a complex semisimple Lie group.
He showed that the corresponding Weyl groups are generated by complex reflections, and obtained the corollary that
the algebra of invariants is a polynomial algebra: it is freely generated by a finite set of fundamental
invariants.
In particular, he embedded the Lie algebra
$\mathfrak{sl}_{333}(\mathbb{C}) = \mathfrak{sl}_3(\mathbb{C}) \oplus \mathfrak{sl}_3(\mathbb{C}) \oplus \mathfrak{sl}_3(\mathbb{C})$
into an exceptional Lie algebra of type $E_6$, and deduced that the algebra of invariants for the natural representation of $\mathfrak{sl}_{333}(\mathbb{C})$
is freely generated by homogeneous polynomials in degrees 6, 9 and 12.
A general computational framework for studying $\theta$-groups has been developed recently by de Graaf
\cite{deGraaf}.

An expression for the fundamental invariant of degree 9 in terms of $3 \times 3$ slices was obtained by Strassen \cite[Lemma 4.5]{Strassen}.
A geometric interpretation of Strassen's result has been given by Ottaviani \cite[\S 3]{Ottaviani}; see also
Domokos and Drensky \cite{DomokosDrensky}.
Nurmiev \cite{Nurmiev} obtained an implicit description of all three invariants in terms of convolutions of volume forms.
An approach to these polynomials using classical invariant theory has been given by Briand et al.~\cite{BLTV}.
For an application to quantum information theory, see Duff and Ferrara \cite{DF}.
For recent related work on $2 \times 2 \times 2 \times 2$ arrays, see Huggins et al.~\cite{Huggins}.

In this paper we obtain explicit expressions for the fundamental invariants in degrees 6, 9, 12 for
the action of $SL_3(\mathbb{C}) \times SL_3(\mathbb{C}) \times SL_3(\mathbb{C})$ on $3 \times 3 \times 3$ arrays.
These invariants have respectively 1152, 9216 and 209061 terms.
We express these homogeneous polynomials in terms of the orbits of the finite symmetry group
$( S_3 \times S_3 \times S_3 ) \rtimes S_3$
acting on monomials of weight zero for the action of $\mathfrak{sl}_{333}(\mathbb{C})$.
We use this action to efficiently represent polynomials as orbit sums;
the invariants, especially in degree 12, involve far too many monomials to be written down completely.
Using this method, it becomes possible to represent the fundamental invariants in
a paper (and not just on a hard disk).
Files containing the invariants are available as ancillary files attached to the \texttt{arXiv} version of this paper.


\section{Preliminaries}

\subsection{Homogeneous polynomials on $3 \times 3 \times 3$ arrays}

Let $e_1, e_2, e_3$ be the standard basis vectors of $\mathbb{C}^3$.
We consider the tensor product
$\mathbb{C}^{333} = \mathbb{C}^3 \otimes \mathbb{C}^3 \otimes \mathbb{C}^3$.
Every element of $\mathbb{C}^{333}$ is a finite sum of elements of the form
$u \otimes v \otimes w$ where $u, v, w \in \mathbb{C}^3$.
A basis for $\mathbb{C}^{333}$ consists of the 27 simple tensors
$e_{ijk} = e_i \otimes e_j \otimes e_k$ for $1 \le i, j, k \le 3$.
We therefore identify the array $X = ( x_{ijk} )$ with the element
  \[
  X
  =
  \sum_{i=1}^3
  \sum_{j=1}^3
  \sum_{k=1}^3
  x_{ijk}
  e_{ijk},
  \qquad
  x_{ijk} \in \mathbb{C}.
  \]
We introduce variables $x_{ijk}$ for $1 \le i, j, k \le 3$ corresponding to the entries of $X$;
this notation is ambiguous but should not be confusing. Strictly speaking,
$x_{ijk}$ is a coordinate function on $\mathbb{C}^{333}$ and is therefore a dual
basis vector $e_i^\ast \otimes e_j^\ast \otimes e_k^\ast$,
but this distinction will not be important for us.
We consider the polynomial algebra $P = \mathbb{C}[ \, x_{ijk} \mid 1 \le i, j, k \le 3 \, ]$.
A basis of $P$ consists of the monomials
  \begin{equation}
  \label{monomial}
  M(E)
  =
  \prod_{i=1}^3 \prod_{j=1}^3 \prod_{k=1}^3
  x_{ijk}^{e_{ijk}}
  =
  x_{111}^{e_{111}}
  \cdots
  x_{ijk}^{e_{ijk}}
  \cdots
  x_{333}^{e_{333}},
  \end{equation}
where $E = ( e_{ijk} )$ is an exponent array of non-negative integers.
The homogeneous subspace $P_N$ of degree $N$ has a basis
consisting of monomials for which
  \[
  \sum_{i=1}^3 \sum_{j=1}^3 \sum_{k=1}^3 e_{ijk} = N.
  \]
A linear operator on $\mathbb{C}^3$ is represented by a $3 \times 3$ matrix $A = ( a_{ij} )$ where $a_{ij} \in \mathbb{C}$.
The action of a triple of invertible operators $(A,B,C)$ on $x_{ijk}$ is given by
  \[
  ( A, B, C ) \cdot x_{ijk}
  =
  \sum_{p=1}^3
  \sum_{q=1}^3
  \sum_{r=1}^3
  a_{pi}
  b_{qj}
  c_{rk}
  x_{pqr}.
  \]
This action extends to polynomials as follows:
  \[
  (A,B,C) \cdot
  f \big( \dots, x_{ijk}, \dots \big)
  =
  f\big(
  \dots,
  (A,B,C) \cdot x_{ijk},
  \dots
  \big).
  \]
We call $f$ an invariant if
$(A,B,C) \cdot f = f$ when $\det(A) = \det(B) = \det(C) = 1$.

\subsection{Representations of Lie algebras}

The $n \times n$ complex matrices of determinant 1 form the special linear group $SL_n(\mathbb{C})$.
Finite-dimensional representations of $SL_n(\mathbb{C})$ can be studied in
terms of the Lie algebra $\mathfrak{sl}_n(\mathbb{C})$, which consists of all
$n \times n$ complex matrices of trace 0. The standard basis of $\mathfrak{sl}_n(\mathbb{C})$
consists of
  \begin{itemize}
  \item the matrix units $U_{i,j}$ for $i \ne j$ with $(i,j)$ entry 1 and
      other entries 0,
  \item the diagonal matrices $H_i = U_{i,i} - U_{i+1,i+1}$ for $i = 1, 2,
      \dots, n{-}1$.
  \end{itemize}
The simple root vectors are the matrix units $T_i = U_{i,i+1}$ for $i = 1, 2,
\dots, n{-}1$. The natural representation of $\mathfrak{sl}_n(\mathbb{C})$ is
its irreducible action on $\mathbb{C}^n$.  We have
  \allowdisplaybreaks
  \begin{align*}
  H_i \cdot e_j
  =
  \begin{cases}
  e_j &\text{if $j = i$} \\
  -e_j &\text{if $j = i{+}1$} \\
  0 &\text{otherwise},
  \end{cases}
  \qquad \qquad
  T_i \cdot e_j
  =
  \begin{cases}
  e_{j-1} &\text{if $j = i{+}1$} \\
  0 &\text{otherwise}.
  \end{cases}
  \end{align*}
We are concerned exclusively with the Lie algebra $\mathfrak{sl}_3(\mathbb{C})$; in this case
  \[
  H_1 =
  \left[ \begin{array}{rrr}
  1 &\!\!\!\!  0 & 0 \\
  0 &\!\!\!\! -1 & 0 \\
  0 &\!\!\!\!  0 & 0
  \end{array} \right]\!,
  \;
  H_2 =
  \left[ \begin{array}{rrr}
  0 & 0 &\!\!\!\!  0 \\
  0 & 1 &\!\!\!\!  0 \\
  0 & 0 &\!\!\!\! -1
  \end{array} \right]\!,
  \;
  T_1 =
  \left[ \begin{array}{rrr}
  0 & 1 & 0 \\
  0 & 0 & 0 \\
  0 & 0 & 0
  \end{array} \right]\!,
  \;
  T_2 =
  \left[ \begin{array}{rrr}
  0 & 0 & 0 \\
  0 & 0 & 1 \\
  0 & 0 & 0
  \end{array} \right]\!.
  \]
The brackets of $H_i$ and $T_j$ are
  \begin{equation}
  \label{hebrackets}
  [ H_1, T_1 ] = 2 T_1,
  \qquad
  [ H_1, T_2 ] = - T_2,
  \qquad
  [ H_2, T_1 ] = - T_1,
  \qquad
  [ H_2, T_2 ] = 2 T_2.
  \end{equation}
We consider the action of
$
\mathfrak{sl}_{333}(\mathbb{C}) =
\mathfrak{sl}_3(\mathbb{C}) \oplus
\mathfrak{sl}_3(\mathbb{C}) \oplus
\mathfrak{sl}_3(\mathbb{C})
$
on its irreducible representation $\mathbb{C}^{333}$.
For $\ell = 1, 2, 3$ we write a superscript $(\ell)$ to indicate the basis elements of the $\ell$-th summand.
These elements act on $x_{ijk}$ as follows ($m = 1, 2$):
  \allowdisplaybreaks
  \begin{alignat*}{2}
  H^{(1)}_m \cdot x_{ijk}
  &=
  \begin{cases}
  x_{ijk} &\text{if $i = m$} \\
  -x_{ijk} &\text{if $i = m{+}1$} \\
  0 &\text{otherwise}
  \end{cases}
  &\qquad
  T^{(1)}_m \cdot x_{ijk}
  &=
  \begin{cases}
  x_{i-1,j,k} &\text{if $i = m{+}1$} \\
  0 &\text{otherwise}
  \end{cases}
  \\
  H^{(2)}_m \cdot x_{ijk}
  &=
  \begin{cases}
  x_{ijk} &\text{if $j = m$} \\
  -x_{ijk} &\text{if $j = m{+}1$} \\
  0 &\text{otherwise}
  \end{cases}
  &\qquad
  T^{(2)}_m \cdot x_{ijk}
  &=
  \begin{cases}
  x_{i,j-1,k} &\text{if $j = m{+}1$} \\
  0 &\text{otherwise}
  \end{cases}
  \\
  H^{(3)}_m \cdot x_{ijk}
  &=
  \begin{cases}
  x_{ijk} &\text{if $k = m$} \\
  -x_{ijk} &\text{if $k = m{+}1$} \\
  0 &\text{otherwise}
  \end{cases}
  &\qquad
  T^{(3)}_m \cdot x_{ijk}
  &=
  \begin{cases}
  x_{i,j,k-1} &\text{if $k = m{+}1$} \\
  0 &\text{otherwise}
  \end{cases}
  \end{alignat*}
The action of a Lie algebra $L$ on a tensor product $V \otimes W$ of
$L$-modules is given by
$A \cdot ( v \otimes w ) = ( A \cdot v ) \otimes w + v \otimes ( A \cdot w )$
for
$A \in L$, $v \in V$, $w \in W$.
If we identify the $N$-th symmetric power $S^N V$ of the $p$-dimensional
$L$-module $V$ with the space of homogeneous polynomials of degree $N$ on a
basis $v_1, \dots, v_p$ of $V$, then the action of $L$ on $S^N V$ is given by
  \allowdisplaybreaks
  \begin{align*}
  &
  A \cdot ( v_1^{e_1} v_2^{e_2} \cdots v_p^{e_p} )
  =
  \sum_{i=1}^p v_1^{e_1} \cdots ( A \cdot v_i^{e_i} ) \cdots v_p^{e_p}
  \\
  &=
  \sum_{i=1}^p v_1^{e_1} \cdots \big( e_i v_i^{e_i-1} ( A \cdot v_i ) \big) \cdots v_p^{e_p}
  =
  \sum_{i=1}^p e_i \, v_1^{e_1} \cdots v_i^{e_i-1} \cdots v_p^{e_p} ( A \cdot v_i ).
  \end{align*}

\begin{lemma} \label{weightlemma}
For $m = 1, 2$ we have
  \allowdisplaybreaks
  \begin{align*}
  &
  H^{(1)}_m \cdot
  \big(
  x_{111}^{e_{111}}
  \cdots
  x_{333}^{e_{333}}
  \big)
  =
  \sum_{j=1}^3 \sum_{k=1}^3
  ( e_{m,j,k} - e_{m+1,j,k} )
  \,
  x_{111}^{e_{111}}
  \cdots
  x_{333}^{e_{333}},
  \\
  &
  H^{(2)}_m \cdot
  \big(
  x_{111}^{e_{111}}
  \cdots
  x_{333}^{e_{333}}
  \big)
  =
  \sum_{i=1}^3 \sum_{k=1}^3
  ( e_{i,m,k} - e_{i,m+1,k} )
  \,
  x_{111}^{e_{111}}
  \cdots
  x_{333}^{e_{333}},
  \\
  &
  H^{(3)}_m \cdot
  \big(
  x_{111}^{e_{111}}
  \cdots
  x_{333}^{e_{333}}
  \big)
  =
  \sum_{i=1}^3 \sum_{j=1}^3
  ( e_{i,j,m} - e_{i,j,m+1} )
  \,
  x_{111}^{e_{111}}
  \cdots
  x_{333}^{e_{333}}.
  \end{align*}
\end{lemma}

The eigenvalues of $x_{111}^{e_{111}} \cdots x_{333}^{e_{333}}$ for the elements $H^{(\ell)}_m$ will be denoted
  \allowdisplaybreaks
  \begin{alignat*}{3}
  \omega_{11} &= \sum_{j=1}^3 \sum_{k=1}^3 ( e_{1jk} {-} e_{2jk} ),
  &\quad
  \omega_{12} &= \sum_{j=1}^3 \sum_{k=1}^3 ( e_{2jk} {-} e_{3jk} ),
  &\quad
  \omega_{21} &= \sum_{i=1}^3 \sum_{k=1}^3 ( e_{i1k} {-} e_{i2k} ),
  \\
  \omega_{22} &= \sum_{i=1}^3 \sum_{k=1}^3 ( e_{i2k} {-} e_{i3k} ),
  &\quad
  \omega_{31} &= \sum_{i=1}^3 \sum_{j=1}^3 ( e_{ij1} {-} e_{ij2} ),
  &\quad
  \omega_{32} &= \sum_{i=1}^3 \sum_{j=1}^3 ( e_{ij2} {-} e_{ij3} ).
  \end{alignat*}
The weight of $x_{111}^{e_{111}} \cdots x_{333}^{e_{333}}$ is the tuple
$(\omega_{11},\omega_{12},\omega_{21},\omega_{22},\omega_{31},\omega_{32})$.
A monomial has weight zero if
$\omega_{\ell m} = 0$ for $\ell = 1, 2, 3$ and $m = 1, 2$.
For given degree and given weight, the weight space
$W( N \mid \omega_{11}, \omega_{12}, \omega_{21}, \omega_{22}, \omega_{31}, \omega_{32} )$
is the subspace of $P_N$ spanned by the corresponding monomials.

\begin{lemma} \label{actionlemma}
For $m = 1, 2$ we have
  \allowdisplaybreaks
  \begin{align*}
  &
  T^{(1)}_m \cdot
  \big(
  x_{111}^{e_{111}}
  \cdots
  x_{333}^{e_{333}}
  \big)
  =
  \sum_{j=1}^3 \sum_{k=1}^3
  \,
  e_{m+1,j,k}
  \,
  x_{111}^{e_{111}}
  \cdots
  x_{m,j,k}^{e_{m,j,k}+1}
  \cdots
  x_{m+1,j,k}^{e_{m+1,j,k}-1}
  \cdots
  x_{333}^{e_{333}},
  \\
  &
  T^{(2)}_m \cdot
  \big(
  x_{111}^{e_{111}}
  \cdots
  x_{333}^{e_{333}}
  \big)
  =
  \sum_{i=1}^3 \sum_{k=1}^3
  \,
  e_{i,m+1,k}
  \,
  x_{111}^{e_{111}}
  \cdots
  x_{i,m,k}^{e_{i,m,k}+1}
  \cdots
  x_{i,m+1,k}^{e_{i,m+1,k}-1}
  \cdots
  x_{333}^{e_{333}},
  \\
  &
  T^{(3)}_m \cdot
  \big(
  x_{111}^{e_{111}}
  \cdots
  x_{333}^{e_{333}}
  \big)
  =
  \sum_{i=1}^3 \sum_{j=1}^3
  \,
  e_{i,j,m+1}
  \,
  x_{111}^{e_{111}}
  \cdots
  x_{i,j,m}^{e_{i,j,m}+1}
  \cdots
  x_{i,j,m+1}^{e_{i,j,m+1}-1}
  \cdots
  x_{333}^{e_{333}}.
  \end{align*}
\end{lemma}

By \eqref{hebrackets} the elements $T^{(\ell)}_m$ induce linear maps:
  \[
  \begin{array}{l}
  T^{(1)}_1 \colon W( N \mid 0, 0, 0, 0, 0, 0 ) \longrightarrow W( N \mid 2, -1, 0, 0, 0, 0 ),
  \\
  T^{(1)}_2 \colon W( N \mid 0, 0, 0, 0, 0, 0 ) \longrightarrow W( N \mid -1, 2, 0, 0, 0, 0 ),
  \\
  T^{(2)}_1 \colon W( N \mid 0, 0, 0, 0, 0, 0 ) \longrightarrow W( N \mid 0, 0, 2, -1, 0, 0 ),
  \\
  T^{(2)}_2 \colon W( N \mid 0, 0, 0, 0, 0, 0 ) \longrightarrow W( N \mid 0, 0, -1, 2, 0, 0 ),
  \\
  T^{(3)}_1 \colon W( N \mid 0, 0, 0, 0, 0, 0 ) \longrightarrow W( N \mid 0, 0, 0, 0, 2, -1 ),
  \\
  T^{(3)}_2 \colon W( N \mid 0, 0, 0, 0, 0, 0 ) \longrightarrow W( N \mid 0, 0, 0, 0, -1, 2 ).
  \end{array}
  \]
We write $\Omega_{\ell m}$ for the nonzero weights and form the direct sum:
  \[
  \Lambda_N = ( T^{(1)}_1, \dots, T^{(3)}_2 ) \colon
  W( N \mid 0, 0, 0, 0, 0, 0 )
  \longrightarrow
  \bigoplus_{\ell=1}^3 \bigoplus_{m=1}^2 W( N \mid \Omega_{\ell m} ).
  \]

\begin{lemma}
The invariant polynomials in degree $N$ for the action of $\mathfrak{sl}_{333}(\mathbb{C})$ on the space
$\mathbb{C}^{333}$ are the (nonzero) elements of the kernel of the linear map $\Lambda_N$.
\end{lemma}

\begin{theorem} \label{vinbergtheorem}
The algebra of invariants for $\mathfrak{sl}_{333}(\mathbb{C})$ acting on $\mathbb{C}^{333}$
is freely generated by three fundamental invariants in degrees $6$, $9$ and $12$.
Thus for $N = 6$, $9$, $12$ the nullspace of $\Lambda_N$ has dimension $1$, $1$, $2$ respectively.
\end{theorem}

\begin{proof}
Vinberg \cite{Vinberg}; see especially item 2 in the table on page 491.
\end{proof}

\subsection{Combinatorics of monomials}

Let $E = ( e_{ijk} )$ be a $3 \times 3 \times 3$ exponent array
with non-negative integer entries corresponding to the monomial \eqref{monomial}.

\begin{definition} \label{monomialorder}
The \textbf{flattening} of $E$ is the ordered list
obtained by applying lexicographical order to the subscripts:
  \[
  \mathrm{flatten}(E) =
  [ e_{111}, e_{112}, e_{113}, e_{121}, e_{122}, e_{123}, \dots, e_{331}, e_{332}, e_{333} ].
  \]
The \textbf{total order} on exponent arrays is the lex order on flattenings:
that is, $E < E'$ for
$\mathrm{flatten}(E) = [ f_1, f_2, \dots, f_{27} ]$
and
$\mathrm{flatten}(E') = [ f'_1, f'_2, \dots, f'_{27} ]$
if and only if $f_i < f'_i$ where $i$ is the least index with $f_i \ne f'_i$.
The \textbf{matrix form} of $E$ is:
  \[
  \left[ \begin{array}{ccc|ccc|ccc}
  e_{111} & e_{121} & e_{131} & e_{112} & e_{122} & e_{132} & e_{113} & e_{123} & e_{133} \\
  e_{211} & e_{221} & e_{231} & e_{212} & e_{222} & e_{232} & e_{213} & e_{223} & e_{233} \\
  e_{311} & e_{321} & e_{331} & e_{212} & e_{222} & e_{232} & e_{213} & e_{223} & e_{233}
  \end{array} \right].
  \]
The third index distinguishes the frontal slices, where a \textbf{slice} is a $3 \times 3$ matrix obtained by
fixing one subscript; there are three parallel slices in each direction.
We call $E$ an \textbf{equal parallel slice array} if
parallel slices have the same sum.
\end{definition}

\begin{lemma}
A basis for $W( N \mid 0,0,0,0,0,0 )$ consists of the monomials of degree $N$
whose exponent arrays are equal parallel slice arrays.
In degree $N$, there are no monomials of weight zero, and hence no invariants,
unless $N$ is a multiple of $3$.
\end{lemma}

\begin{proof}
This follows immediately from Lemma \ref{weightlemma}.
\end{proof}

To generate the exponent arrays for weight zero monomials of degree $N$,
we use nested loops and the condition that the horizontal parallel slices
have sum $N/3$.
(See Table \ref{generateweightzero}.)
To generate the higher weight monomials, we first use Lemma \ref{actionlemma} to generate the monomials
for weights $\Omega_{1m}$ ($m = 1,2$).
We then use symmetry to obtain the monomials for weights $\Omega_{\ell m}$ ($\ell = 2,3$, $m = 1,2$):
if $E$ is an exponent array of weight $\Omega_{1m}$ then $E'$
defined by $e'_{ijk} = e_{jik}$ is an exponent array of weight $\Omega_{2m}$,
and similarly for monomials of weight $\Omega_{3m}$.
(See Table \ref{generatehigherweights}.)

\begin{table}
\begin{itemize}
\item
set $\texttt{weightzeromonomials} \leftarrow [\;]$ \emph{(empty list)}
\item
for $f_1$ from 0 to $N/3$ do for $f_2$ from 0 to $N/3-f_1$ do \dots
\item[]
for $f_8$ from 0 to $N/3-(f_1+\cdots+f_7)$ do:
\begin{itemize}
\item[$\cdot$]
set $f_9 \leftarrow N/3-(f_1+\cdots+f_8)$
\item[$\cdot$]
for $f_{10}$ from 0 to $N/3$ do for $f_{11}$ from 0 to $N/3-f_{10}$ do \dots
\item[]
for $f_{17}$ from 0 to $N/3-(f_{10}+\cdots+f_{16})$ do:
\begin{itemize}
\item[$\cdot$]
set $f_{18} \leftarrow N/3-(f_{10}+\cdots+f_{17})$
\item[$\cdot$]
for $f_{19}$ from 0 to $N/3$ do for $f_{20}$ from 0 to $N/3-f_{19}$ do \dots
\item[]
for $f_{26}$ from 0 to $N/3-(f_{19}+\cdots+f_{25})$ do:
\begin{itemize}
\item[$\cdot$]
set $f_{27} \leftarrow N/3-(f_{19}+\cdots+f_{26})$
\item[$\cdot$]
set $e \leftarrow \texttt{unflatten}( f )$
\item[$\cdot$]
if $\sum_{j,k} e_{1jk} = \sum_{j,k} e_{2jk} = \sum_{j,k} e_{3jk}$ and
\item[]
$\sum_{i,k} e_{i1k} = \sum_{i,k} e_{i2k} = \sum_{i,k} e_{i3k}$ and
\item[]
$\sum_{i,j} e_{ij1} = \sum_{i,j} e_{ij2} = \sum_{i,j} e_{ij3}$ then
\item[]
append $[f_1,\dots,f_{27}]$ to $\texttt{weightzeromonomials}$
\end{itemize}
\end{itemize}
\end{itemize}
\item
return $\texttt{weightzeromonomials}$
\end{itemize}
\medskip
\caption{Pseudocode to generate weight zero monomials of degree $N$}
\label{generateweightzero}
\end{table}

\begin{table}
\begin{itemize}
\item
set $\ell \leftarrow 1$
\item
for $m = 1,2$ do
\begin{itemize}
\item[$\cdot$]
set $\texttt{higherweightmonomials}[ \ell, m ] \leftarrow [\;]$ \emph{(empty list)}
\item[$\cdot$]
for $x$ in $\texttt{weightzeromonomials}$ do
\begin{itemize}
\item[$\cdot$]
set $e \leftarrow \texttt{unflatten}( x )$
\item[$\cdot$]
for $j = 1,2,3$ do for $k = 1,2,3$ do
\begin{enumerate}
\item[$\cdot$]
if $e_{m+1,j,k} > 0$ then
\begin{enumerate}
\item[]
set $f \leftarrow e$
\item[]
set $f_{m+1,j,k} \leftarrow f_{m+1,j,k} - 1$
\item[]
set $f_{m,j,k} \leftarrow f_{m,j,k} + 1$
\item[]
append $\texttt{flatten}( f )$ to $\texttt{higherweightmonomials}[ \ell, m ]$
\end{enumerate}
\end{enumerate}
\end{itemize}
\end{itemize}
\item
return $\texttt{higherweightmonomials}[ \ell, 1 ]$, $\texttt{higherweightmonomials}[ \ell, 2 ]$
\end{itemize}
\medskip
\caption{Pseudocode to generate higher weight monomials ($\ell=1$)}
\label{generatehigherweights}
\end{table}

\begin{definition}
The \textbf{symmetry group} acting on the weight zero monomials in degree $N$
is the semidirect product
$G = ( S_3 \times S_3 \times S_3 ) \rtimes S_3$.
Each factor in $S_3 \times S_3 \times S_3$ permutes the parallel slices in the corresponding direction;
the last $S_3$ permutes the directions.
More precisely, $( \alpha, \beta, \gamma )$ and $\delta$ act on
$E = ( e_{ijk} )$ by:
  \[
  \big( ( \alpha, \beta, \gamma ) \cdot E \big)_{i,j,k}
  =
  e_{ \alpha(i), \beta(j), \gamma(k) },
  \qquad
  \big( \delta \cdot E \big)_{i,j,k}
  =
  e_{i^\delta\!, j^\delta\!, k^\delta}.
  \]
The \textbf{orbit} of a weight zero monomial $M(E)$ is
  $
  \mathcal{O}(M)
  =
  \{ \, M( g \cdot E ) \mid g \in G \, \}
  $.
The \textbf{symmetric} and \textbf{alternating orbit sums} are defined by
  \allowdisplaybreaks
  \begin{align*}
  \mathcal{O}^+(M)
  &=
  \frac{|\mathcal{O}(M)|}{|G|}
  \sum_{g \in G} M( g \cdot E ),
  \qquad
  \mathcal{O}^-(M)
  =
  \frac{|\mathcal{O}(M)|}{|G|}
  \sum_{g \in G} \epsilon(g) \, M( g \cdot E ),
  \end{align*}
where $\epsilon(g)$ is the product of the signs of the components of
$g = (\alpha,\beta,\gamma,\delta) \in G$.
For some monomials, the alternating orbit sum will be zero.
\end{definition}

See Table \ref{generateorbits} for procedures to generate the orbit of a weight zero monomial;
\texttt{smallorbit} permutes the slices,
and \texttt{largeorbit} permutes the directions.

\begin{table}
\begin{itemize}
\item
$\texttt{smallgrouporbit}( f )$
\begin{itemize}
\item[$\cdot$]
set $e \leftarrow \texttt{unflatten}( f )$
\item[$\cdot$]
set $\texttt{orbit} \leftarrow \{\,\}$
\item[$\cdot$]
for $p \in S_3$ do for $q \in S_3$ do for $r \in S_3$ do:
\begin{itemize}
\item[$\cdot$]
for $i = 1,2,3$ do for $j = 1,2,3$ do for $k = 1,2,3$ do:
\item[] \qquad
set $f_{i,j,k} \leftarrow e_{p(i),q(j),r(k)}$
\item[$\cdot$]
set $\texttt{orbit} \leftarrow \texttt{orbit} \cup \{ \texttt{flatten}(f) \}$
\end{itemize}
\item[$\cdot$]
return $\texttt{orbit}$
\end{itemize}
\bigskip
\item
$\texttt{largegrouporbit}( f )$
\begin{itemize}
\item[$\cdot$]
set $e \leftarrow \texttt{unflatten}( f )$
\item[$\cdot$]
set $\texttt{orbit} \leftarrow \{\,\}$
\item[$\cdot$]
for $s \in S_3$ do:
\begin{itemize}
\item[$\cdot$]
for $i = 1,2,3$ do for $j = 1,2,3$ do for $k = 1,2,3$ do:
\begin{itemize}
\item[]
set $t \leftarrow (i,j,k)$
\item[]
set $f_{i,j,k} \leftarrow e_{t_{s(1)},t_{s(2)},t_{s(3)}}$
\end{itemize}
\item[$\cdot$]
$\texttt{orbit} \leftarrow \texttt{orbit} \cup \texttt{smallgrouporbit}( \texttt{flatten}( f ) )$
\end{itemize}
\item[$\cdot$]
return $\texttt{orbit}$
\end{itemize}
\end{itemize}
\medskip
\caption{Pseudocode to generate the orbits of a weight zero monomial}
\label{generateorbits}
\end{table}


\section{Degree 6}

\begin{lemma} \label{d6w0lemma}
In degree $6$, there are $1152$ monomials of weight zero,
and $792$ monomials of each higher weight $\Omega_{\ell m}$ for $\ell = 1, 2, 3$ and $m = 1, 2$.
\end{lemma}

\begin{proof}
We generate the monomials using an implementation in Maple 15 of the algorithms in Tables
\ref{generateweightzero} and \ref{generatehigherweights} with $N = 6$.
\end{proof}

\begin{lemma} \label{d6lemma}
The nullspace of the matrix representing $\Lambda_6$ has dimension $1$ modulo $p = 101$.
Using symmetric representatives, the canonical basis vector
of the nullspace has coefficients $\{ -10, -4, -2, 1, 2, 4, 8 \}$.
If we interpret these as integers then the corresponding polynomial is
an invariant for the action of $\mathfrak{sl}_{333}(\mathbb{C})$.
\end{lemma}

\begin{proof}
We use the \texttt{LinearAlgebra[Modular]} package in Maple 15 to create a matrix $B$
with an upper block of size $1152 \times 1152$
and a lower block of size $792 \times 1152$.
For $\ell = 1, 2, 3$ and $m = 1, 2$ we apply Lemma \ref{actionlemma}
to store the matrix representing
$T^{(\ell)}_m$ in the lower block; we then compute the row canonical form.
See Table \ref{fillmatrix} for $\ell = 1$; $\texttt{monomialindex}$ does
a binary search for a higher weight monomial in the lexicographically
ordered list of all monomials of that weight.
At termination, $B$ has rank 1151.
From the row canonical form, we extract a basis vector for the nullspace.
We then perform another computation using integer arithmetic to verify
that the corresponding polynomial is indeed an invariant over $\mathbb{C}$.
\end{proof}

\begin{table}
\begin{itemize}
\item
set $B \leftarrow \text{Matrix}( 1944, 1152 )$ modulo 101
\item
set $\ell \leftarrow 1$
\item
set $c \leftarrow 0$
\item
for $x \in \texttt{weightzeromonomials}$ do
\begin{itemize}
\item[$\cdot$]
set $c \leftarrow c + 1$
\item[$\cdot$]
for $m = 1,2$ do
\begin{itemize}
\item[$\cdot$]
set $e \leftarrow \texttt{unflatten}( x )$
\item[$\cdot$]
for $j = 1,2,3$ do for $k = 1,2,3$ do
\begin{itemize}
\item[$\cdot$]
if $e_{m+1,j,k} > 0$ then
\begin{enumerate}
\item[]
set $f \leftarrow e$
\item[]
set $f_{m+1,j,k} \leftarrow f_{m+1,j,k} - 1$
\item[]
set $f_{m,j,k} \leftarrow f_{m,j,k} + 1$
\item[]
set $g \leftarrow \texttt{flatten}( f )$
\item[]
set $r \leftarrow \texttt{monomialindex}( g, \ell, m )$:
\item[]
set $B_{1152+r,c} \leftarrow ( B_{1152+r,c} + e_{m+1,j,k} )$ modulo 101
\end{enumerate}
\end{itemize}
\end{itemize}
\end{itemize}
\item
compute the row canonical form of $B$
\end{itemize}
\medskip
\caption{Pseudocode to fill the matrix ($N = 6$, $\ell=1$)}
\label{fillmatrix}
\end{table}

\begin{lemma} \label{d6orbitlemma}
The coefficients of the canonical basis vector for the nullspace of $\Lambda_6$
are constant on orbits for the action of $( S_3 \times S_3 \times S_3 ) \rtimes S_3$ on
weight zero monomials.  For each orbit, the coefficient,
the matrix form of the minimal representative, and
the orbit size, are displayed in Table
$\ref{degree6invariant}$.
\end{lemma}

\begin{proof}
We implemented the algorithms of Tables \ref{generateorbits} and \ref{fillmatrix} in Maple 15.
\end{proof}

\begin{table} \tiny
  \[
  \begin{array}{cccc}
  I_6 &\qquad
  \text{MINIMAL REPRESENTATIVE} &\qquad \text{ORBIT SIZE} &\qquad \#
  \\
  \midrule
  -10  &\qquad
  \left[ \begin{array}{ccc|ccc|ccc}
  0 & 0 & 0 & 0 & 0 & 1 & 0 & 1 & 0 \\
  0 & 0 & 1 & 0 & 0 & 0 & 1 & 0 & 0 \\
  0 & 1 & 0 & 1 & 0 & 0 & 0 & 0 & 0 \\
  \end{array} \right] &\qquad
  36  &\qquad
  1
  \\[12pt]
  -4  &\qquad
  \left[ \begin{array}{ccc|ccc|ccc}
  0 & 0 & 0 & 0 & 0 & 1 & 0 & 0 & 1 \\
  0 & 1 & 0 & 0 & 0 & 0 & 1 & 0 & 0 \\
  1 & 0 & 0 & 0 & 1 & 0 & 0 & 0 & 0 \\
  \end{array} \right] &\qquad
  324  &\qquad
  2
  \\[12pt]
  -2  &\qquad
  \left[ \begin{array}{ccc|ccc|ccc}
  0 & 0 & 0 & 0 & 0 & 0 & 0 & 0 & 2 \\
  0 & 1 & 0 & 0 & 1 & 0 & 0 & 0 & 0 \\
  1 & 0 & 0 & 1 & 0 & 0 & 0 & 0 & 0 \\
  \end{array} \right] &\qquad
  162  &\qquad
  3
  \\[12pt]
  1  &\qquad
  \left[ \begin{array}{ccc|ccc|ccc}
  0 & 0 & 0 & 0 & 0 & 0 & 0 & 0 & 2 \\
  0 & 0 & 0 & 0 & 2 & 0 & 0 & 0 & 0 \\
  2 & 0 & 0 & 0 & 0 & 0 & 0 & 0 & 0 \\
  \end{array} \right] &\qquad
  36  &\qquad
  4
  \\[12pt]
  2  &\qquad
  \left[ \begin{array}{ccc|ccc|ccc}
  0 & 0 & 0 & 0 & 0 & 1 & 0 & 0 & 1 \\
  0 & 1 & 0 & 0 & 0 & 0 & 0 & 1 & 0 \\
  1 & 0 & 0 & 1 & 0 & 0 & 0 & 0 & 0 \\
  \end{array} \right] &\qquad
  108  &\qquad
  5
  \\[12pt]
  2  &\qquad
  \left[ \begin{array}{ccc|ccc|ccc}
  0 & 0 & 0 & 0 & 0 & 1 & 0 & 0 & 1 \\
  0 & 0 & 0 & 0 & 1 & 0 & 1 & 0 & 0 \\
  1 & 1 & 0 & 0 & 0 & 0 & 0 & 0 & 0 \\
  \end{array} \right] &\qquad
  324  &\qquad
  6
  \\[12pt]
  4  &\qquad
  \left[ \begin{array}{ccc|ccc|ccc}
  0 & 0 & 0 & 0 & 0 & 0 & 0 & 0 & 2 \\
  0 & 1 & 0 & 1 & 0 & 0 & 0 & 0 & 0 \\
  1 & 0 & 0 & 0 & 1 & 0 & 0 & 0 & 0 \\
  \end{array} \right] &\qquad
  54  &\qquad
  7
  \\[12pt]
  8  &\qquad
  \left[ \begin{array}{ccc|ccc|ccc}
  0 & 0 & 0 & 0 & 0 & 1 & 0 & 1 & 0 \\
  0 & 0 & 1 & 0 & 0 & 0 & 1 & 0 & 0 \\
  1 & 0 & 0 & 0 & 1 & 0 & 0 & 0 & 0 \\
  \end{array} \right] &\qquad
  108  &\qquad
  8
  \\
  \midrule
  \end{array}
  \]
\caption{The fundamental invariant in degree 6}
\label{degree6invariant}
\end{table}

\begin{theorem}
Every invariant in degree $6$ for
$SL_3(\mathbb{C}) \times SL_3(\mathbb{C}) \times SL_3(\mathbb{C})$
acting on $3 \times 3 \times 3$ arrays $X = ( x_{ijk} )$ is a scalar multiple of
  \begin{align*}
  &
  -10 \,
  \mathcal{O}^+
  \big( \,
  x_{123}
  x_{132}
  x_{213}
  x_{231}
  x_{312}
  x_{321}
  \big)
  -4 \,
  \mathcal{O}^+
  \big( \,
  x_{132}
  x_{133}
  x_{213}
  x_{221}
  x_{311}
  x_{322}
  \big)
  \\
  &
  -2 \,
  \mathcal{O}^+
  \big( \,
  x_{133}^2
  x_{221}
  x_{222}
  x_{311}
  x_{312}
  \big)
  + \,
  \mathcal{O}^+
  \big( \,
  x_{133}^2
  x_{222}^2
  x_{311}^2
  \big)
  \\
  &
  +2 \,
  \mathcal{O}^+
  \big( \,
  x_{132}
  x_{133}
  x_{221}
  x_{223}
  x_{311}
  x_{312}
  \big)
  +2 \,
  \mathcal{O}^+
  \big( \,
  x_{132}
  x_{133}
  x_{213}
  x_{222}
  x_{311}
  x_{321}
  \big)
  \\
  &
  +4 \,
  \mathcal{O}^+
  \big( \,
  x_{133}^2
  x_{212}
  x_{221}
  x_{311}
  x_{322}
  \big)
  +8 \,
  \mathcal{O}^+
  \big( \,
  x_{123}
  x_{132}
  x_{213}
  x_{231}
  x_{311}
  x_{322}
  \big).
  \end{align*}
\end{theorem}

\begin{proof}
This is a restatement of the results in Table \ref{degree6invariant}.
\end{proof}


\begin{table} \tiny
  \[
  \begin{array}{cccc}
  I_9 &\qquad
  \text{MINIMAL REPRESENTATIVE} &\qquad \text{ORBIT SIZE} &\qquad \#
  \\
  \midrule
  1  &\qquad
  \left[ \begin{array}{ccc|ccc|ccc}
  0 & 0 & 0 & 0 & 0 & 1 & 0 & 1 & 1 \\
  0 & 1 & 0 & 0 & 0 & 1 & 1 & 0 & 0 \\
  2 & 0 & 0 & 0 & 1 & 0 & 0 & 0 & 0 \\
  \end{array} \right] &\qquad
  648  &\qquad
  1
  \\[12pt]
  1  &\qquad
  \left[ \begin{array}{ccc|ccc|ccc}
  0 & 0 & 0 & 0 & 0 & 1 & 0 & 1 & 1 \\
  0 & 1 & 1 & 0 & 0 & 0 & 1 & 0 & 0 \\
  0 & 1 & 0 & 2 & 0 & 0 & 0 & 0 & 0 \\
  \end{array} \right] &\qquad
  648  &\qquad
  2
  \\[12pt]
  1  &\qquad
  \left[ \begin{array}{ccc|ccc|ccc}
  0 & 0 & 0 & 0 & 0 & 1 & 0 & 1 & 1 \\
  0 & 1 & 0 & 0 & 1 & 0 & 1 & 0 & 0 \\
  1 & 0 & 1 & 1 & 0 & 0 & 0 & 0 & 0 \\
  \end{array} \right] &\qquad
  1296  &\qquad
  3
  \\[12pt]
  1  &\qquad
  \left[ \begin{array}{ccc|ccc|ccc}
  0 & 0 & 0 & 0 & 0 & 1 & 0 & 1 & 1 \\
  0 & 2 & 0 & 0 & 0 & 0 & 1 & 0 & 0 \\
  1 & 0 & 0 & 1 & 0 & 1 & 0 & 0 & 0 \\
  \end{array} \right] &\qquad
  648  &\qquad
  4
  \\[12pt]
  1  &\qquad
  \left[ \begin{array}{ccc|ccc|ccc}
  0 & 0 & 0 & 0 & 0 & 1 & 0 & 1 & 1 \\
  0 & 0 & 1 & 1 & 0 & 0 & 0 & 1 & 0 \\
  1 & 1 & 0 & 1 & 0 & 0 & 0 & 0 & 0 \\
  \end{array} \right] &\qquad
  648  &\qquad
  5
  \\[12pt]
  1  &\qquad
  \left[ \begin{array}{ccc|ccc|ccc}
  0 & 0 & 0 & 0 & 0 & 1 & 0 & 1 & 1 \\
  0 & 0 & 1 & 1 & 1 & 0 & 0 & 0 & 0 \\
  2 & 0 & 0 & 0 & 0 & 0 & 0 & 1 & 0 \\
  \end{array} \right] &\qquad
  1296  &\qquad
  6
  \\[12pt]
  1  &\qquad
  \left[ \begin{array}{ccc|ccc|ccc}
  0 & 0 & 0 & 0 & 0 & 1 & 0 & 1 & 1 \\
  0 & 1 & 0 & 1 & 0 & 1 & 0 & 0 & 0 \\
  1 & 1 & 0 & 0 & 0 & 0 & 1 & 0 & 0 \\
  \end{array} \right] &\qquad
  648  &\qquad
  7
  \\[12pt]
  1  &\qquad
  \left[ \begin{array}{ccc|ccc|ccc}
  0 & 0 & 0 & 0 & 0 & 1 & 0 & 1 & 1 \\
  0 & 1 & 1 & 1 & 0 & 0 & 0 & 0 & 0 \\
  1 & 0 & 0 & 0 & 1 & 0 & 1 & 0 & 0 \\
  \end{array} \right] &\qquad
  1296  &\qquad
  8
  \\[12pt]
  1  &\qquad
  \left[ \begin{array}{ccc|ccc|ccc}
  0 & 0 & 0 & 0 & 0 & 1 & 0 & 1 & 1 \\
  0 & 1 & 0 & 1 & 0 & 0 & 1 & 0 & 0 \\
  0 & 1 & 1 & 1 & 0 & 0 & 0 & 0 & 0 \\
  \end{array} \right] &\qquad
  216  &\qquad
  9
  \\[12pt]
  1  &\qquad
  \left[ \begin{array}{ccc|ccc|ccc}
  0 & 0 & 0 & 0 & 0 & 1 & 0 & 1 & 1 \\
  0 & 0 & 1 & 2 & 0 & 0 & 0 & 0 & 0 \\
  0 & 2 & 0 & 0 & 0 & 0 & 1 & 0 & 0 \\
  \end{array} \right] &\qquad
  648  &\qquad
  10
  \\[12pt]
  1  &\qquad
  \left[ \begin{array}{ccc|ccc|ccc}
  0 & 0 & 0 & 0 & 0 & 1 & 0 & 1 & 1 \\
  1 & 0 & 0 & 1 & 0 & 1 & 0 & 0 & 0 \\
  1 & 1 & 0 & 0 & 0 & 0 & 0 & 1 & 0 \\
  \end{array} \right] &\qquad
  72  &\qquad
  11
  \\[12pt]
  1  &\qquad
  \left[ \begin{array}{ccc|ccc|ccc}
  0 & 0 & 0 & 0 & 0 & 2 & 0 & 1 & 0 \\
  0 & 1 & 1 & 0 & 0 & 0 & 1 & 0 & 0 \\
  1 & 0 & 0 & 1 & 0 & 0 & 0 & 1 & 0 \\
  \end{array} \right] &\qquad
  432  &\qquad
  12
  \\[12pt]
  1  &\qquad
  \left[ \begin{array}{ccc|ccc|ccc}
  0 & 0 & 0 & 0 & 0 & 2 & 0 & 1 & 0 \\
  0 & 1 & 0 & 0 & 1 & 0 & 1 & 0 & 0 \\
  2 & 0 & 0 & 0 & 0 & 0 & 0 & 0 & 1 \\
  \end{array} \right] &\qquad
  648  &\qquad
  13
  \\[12pt]
  1  &\qquad
  \left[ \begin{array}{ccc|ccc|ccc}
  0 & 0 & 0 & 0 & 0 & 2 & 0 & 1 & 0 \\
  0 & 0 & 1 & 0 & 0 & 0 & 2 & 0 & 0 \\
  0 & 2 & 0 & 1 & 0 & 0 & 0 & 0 & 0 \\
  \end{array} \right] &\qquad
  72  &\qquad
  14
  \\
  \midrule
  \end{array}
  \]
\caption{The fundamental invariant in degree 9}
\label{degree9invariant}
\end{table}


\section{Degree 9}

\begin{lemma}
In degree $9$, there are $22620$ monomials of weight zero,
and $17802$ monomials of each higher weight $\Omega_{\ell m}$
for $\ell = 1, 2, 3$ and $m = 1, 2$.
\end{lemma}

\begin{proof}
Similar to the proof of Lemma \ref{d6w0lemma}.
\end{proof}

\begin{lemma}
The nullspace of the matrix representing $\Lambda_9$ has dimension $1$ modulo $p = 101$.
Using symmetric representatives, the canonical basis vector
of the nullspace has coefficients $\{ -1, 0, 1 \}$;
the nonzero coefficients each occur $4608$ times and $0$ occurs $13404$ times.
If we interpret these as integers then the corresponding polynomial is
an invariant for the action of $\mathfrak{sl}_{333}(\mathbb{C})$.
\end{lemma}

\begin{proof}
Similar to the proof of Lemma \ref{d6lemma}.
\end{proof}

\begin{lemma}
The canonical basis vector for the nullspace of $\Lambda_9$ is a linear combination
of alternating orbit sums for the action of $( S_3 \times S_3 \times S_3 ) \rtimes S_3$ on
weight zero monomials.  For each orbit, the coefficient,
the matrix form of the minimal representative, and
the orbit size, are displayed in Table $\ref{degree9invariant}$.
\end{lemma}

\begin{proof}
Similar to the proof of Lemma \ref{d6orbitlemma}.
(There are another 30 orbits for the action of the
symmetry group but they all occur with coefficient 0.)
\end{proof}

\begin{theorem}
Every invariant in degree $9$ for
$SL_3(\mathbb{C}) \times SL_3(\mathbb{C}) \times SL_3(\mathbb{C})$
acting on $3 \times 3 \times 3$ arrays
$X = ( x_{ijk} )$ is a scalar multiple of
  \begin{align*}
  &\mathcal{O}^-
  \big( \,
  x_{123}
  x_{132}
  x_{133}
  x_{213}
  x_{221}
  x_{232}
  x_{311}^2
  x_{322}
  \big)
  \\
  + \,
  &\mathcal{O}^-
  \big( \,
  x_{123}
  x_{132}
  x_{133}
  x_{213}
  x_{221}
  x_{231}
  x_{312}^2
  x_{321}
  \big)
  \\
  + \,
  &\mathcal{O}^-
  \big( \,
  x_{123}
  x_{132}
  x_{133}
  x_{213}
  x_{221}
  x_{222}
  x_{311}
  x_{312}
  x_{331}
  \big)
  \\
  + \,
  &\mathcal{O}^-
  \big( \,
  x_{123}
  x_{132}
  x_{133}
  x_{213}
  x_{221}^2
  x_{311}
  x_{312}
  x_{332}
  \big)
  \\
  + \,
  &\mathcal{O}^-
  \big( \,
  x_{123}
  x_{132}
  x_{133}
  x_{212}
  x_{223}
  x_{231}
  x_{311}
  x_{312}
  x_{321}
  \big)
  \\
  + \,
  &\mathcal{O}^-
  \big( \,
  x_{123}
  x_{132}
  x_{133}
  x_{212}
  x_{222}
  x_{231}
  x_{311}^2
  x_{323}
  \big)
  \\
  + \,
  &\mathcal{O}^-
  \big( \,
  x_{123}
  x_{132}
  x_{133}
  x_{212}
  x_{221}
  x_{232}
  x_{311}
  x_{313}
  x_{321}
  \big)
  \\
  + \,
  &\mathcal{O}^-
  \big( \,
  x_{123}
  x_{132}
  x_{133}
  x_{212}
  x_{221}
  x_{231}
  x_{311}
  x_{313}
  x_{322}
  \big)
  \\
  + \,
  &\mathcal{O}^-
  \big( \,
  x_{123}
  x_{132}
  x_{133}
  x_{212}
  x_{213}
  x_{221}
  x_{312}
  x_{321}
  x_{331}
  \big)
  \\
  + \,
  &\mathcal{O}^-
  \big( \,
  x_{123}
  x_{132}
  x_{133}
  x_{212}^2
  x_{231}
  x_{313}
  x_{321}^2
  \big)
  \\
  + \,
  &\mathcal{O}^-
  \big( \,
  x_{123}
  x_{132}
  x_{133}
  x_{211}
  x_{212}
  x_{232}
  x_{311}
  x_{321}
  x_{323}
  \big)
  \\
  + \,
  &\mathcal{O}^-
  \big( \,
  x_{123}
  x_{132}^2
  x_{213}
  x_{221}
  x_{231}
  x_{311}
  x_{312}
  x_{323}
  \big)
  \\
  + \,
  &\mathcal{O}^-
  \big( \,
  x_{123}
  x_{132}^2
  x_{213}
  x_{221}
  x_{222}
  x_{311}^2
  x_{333}
  \big)
  \\
  + \,
  &\mathcal{O}^-
  \big( \,
  x_{123}
  x_{132}^2
  x_{213}^2
  x_{231}
  x_{312}
	  x_{321}^2
  \big).
  \end{align*}
\end{theorem}

\begin{proof}
This is a restatement of the results in Table \ref{degree9invariant}.
\end{proof}


\section{Degree 12}

We find explicit forms of two linearly independent invariants in degree 12,
and verify that the simpler invariant can be taken as the generator in degree 12.

\begin{lemma}
In degree $12$, there are $302274$ monomials of weight zero,
and $254961$ monomials of each higher weight $\Omega_{\ell m}$ for $\ell = 1, 2, 3$ and $m = 1, 2$.
\end{lemma}

Even using modular arithmetic, it is impossible to process efficiently a matrix
with 302274 columns and $302274 + 254961 = 557235$ rows.
Therefore we look for invariants that are linear combinations
of the symmetric orbit sums.

\begin{lemma}
In degree $12$, there are $359$ orbits for the action of
$( S_3 \times S_3 \times S_3 ) \rtimes S_3$
on weight zero monomials.
For each orbit, the flattened minimal representative and the orbit size are given in Tables
$\ref{degree12table1}$--$\ref{degree12table5}$.
\end{lemma}

We consider a matrix with 359 columns, one for each symmetric orbit sum.

\begin{lemma} \label{d12modularlemma}
The nullspace of the matrix representing the restriction of $\Lambda_{12}$
to the span of the symmetric orbit sums has dimension $2$ modulo $p = 101$.
\end{lemma}

\begin{proof}
We use modular arithmetic on a matrix $M$ with an upper block of size $359 \times 359$ and
a lower block of size $254961 \times 359$.
For each $\ell = 1, 2, 3$ and $m = 1, 2$ we consider the restriction of
$T^{(\ell)}_m$ to the subspace of $W( 12 \mid 0,0,0,0,0,0 )$ spanned by the symmetric
orbit sums; we store the corresponding matrix in the lower block,
and compute the row canonical form.
After the first iteration ($\ell = m = 1$) the rank is 357, and does not increase for
the remaining five iterations.
\end{proof}

Lemma \ref{d12modularlemma} agrees with Theorem \ref{vinbergtheorem}:
a basis of the invariants in degree 12 consists of $I_6^2$ together with a new generator $I_{12}$.
We confirm this result with further computations using integer arithmetic.

\begin{lemma} \label{smithlemma}
Let $A$ be a matrix with integer entries.
If $r_0$ is its rank over $\mathbb{Q}$ and $r_p$ its its rank over the prime field $\mathbb{F}_p$,
then $r_p \le r_0$.  Hence the dimension of the nullspace
over $\mathbb{Q}$ is no larger than the dimension of the nullspace over $\mathbb{F}_p$.
\end{lemma}

\begin{proof}
The rank of a matrix is $r$ if and only if all $(r{+}1)$-minors are zero, and at least one $r$-minor is not zero.
Since these minors are polynomials in the matrix entries with integer coefficients, the claim follows.
\end{proof}

\begin{theorem}
In degree $12$, the polynomials $I_{12}$ and $I_{12}'$ given in
Tables $\ref{degree12table1}$--$\ref{degree12table5}$ form a reduced basis
for the lattice of invariants  with integer coefficients.
\end{theorem}

\begin{proof}
For each $\ell = 1, 2, 3$ and $m = 1, 2$ we totally order the 254961 higher weight
monomials, and partition each set into 399 blocks of 639 consecutive monomials.
For $j = 1, 2, \dots 359$ we compute the action of $T^{(\ell)}_m$ on the $j$-th symmetric
orbit sum, and separate the resulting terms corresponding to the different blocks.

We create an integer matrix $M$ with an upper block of size $359 \times 359$ and a lower block of size $639 \times 359$,
initialized to zero.
For $\ell = 1, 2, 3$, $m = 1, 2$, $j = 1, 2, \dots, 359$, $k = 1, 2, \dots, 399$
we consider the terms in block $k$ obtained by applying $T^{(\ell)}_m$ to the $j$-th symmetric orbit sum.
We store the integer coefficients of these higher weight monomials in the lower block,
and compute the Hermite normal form (HNF).

We obtain the same behavior as in Lemma \ref{d12modularlemma}:
after the first iteration the rank is 357, and does not increase for the remaining iterations.
At termination, we have an integer matrix of size $357 \times 359$,
also called $M$, whose integer nullspace consists of the coefficient vectors of
the invariant polynomials of degree 12.

To find a lattice basis for this integer nullspace, we compute the HNF of the
transpose $M^t$, obtaining integer matrices $U$ (invertible) and $H$, of sizes
$359 \times 359$ and $359 \times 357$ respectively, for which $U M^t = H$.
The last two rows of $U$ form a lattice basis for the integer nullspace of $M$.
We apply the algorithm of Gauss-Lagrange to these last two rows to find
a short basis of the integer nullspace.
(For an introduction to lattice basis reduction, see Bremner \cite{BremnerBook}.)
Since the lattice is 2-dimensional, the output consists of a
shortest nonzero vector and a shortest vector which is not a multiple
of the first vector.
The components of these reduced basis vectors $I_{12}$ and $I'_{12}$ are
given in Tables \ref{degree12table1}--\ref{degree12table5}.
Finally, we perform an independent check using integer arithmetic that $I_{12}$ and $I'_{12}$
are annihilated by all the $T^{(\ell)}_m$.
\end{proof}

\begin{lemma}
We have $I'_{12} = I_6^2 + 21 \, I_{12}$.
Hence $I_{12}$ is the fundamental invariant in degree 12;
it involves $235$ of the $359$ orbits, with a total of $209061$ monomials.
\end{lemma}

\begin{proof}
We compute the expression for $I_6^2$ as a linear combination of $I_{12}$ and $I'_{12}$.
This can be performed very efficiently in Maple 15 using the algorithms
of Monagan and Pearce \cite{MP1,MP2} which are based on the work of Johnson \cite{J}.
\end{proof}

We can now state our main result, and an open problem.

\begin{theorem}
The three fundamental invariants for $3 \times 3 \times 3$ arrays are:
  \begin{itemize}
  \item
the polynomial $I_6$ of Table $\ref{degree6invariant}$, which is a linear combination of
the $8$ symmetric orbits in degree $6$, and has altogether $1152$ terms;
  \item
the polynomial $I_9$ of Table $\ref{degree9invariant}$, which is the sum of $14$ alternating
orbits in degree $9$, and has altogether $9216$ terms;
  \item
the polynomial $I_{12}$ of Tables $\ref{degree12table1}$--$\ref{degree12table5}$, which is a
linear combination of $235$ symmetric orbits in degree $12$, and has altogether $209061$ terms.
  \end{itemize}
\end{theorem}

\begin{openproblem}
Corollary 3.9 of Gelfand et al.~\cite{GKZ} implies that the hyperdeterminant
of a $3 \times 3 \times 3$ array has degree 36, and hence has the form
  \[
  a \, I_6^6 +
  b \, I_6^4 I_{12} +
  c \, I_6^3 I_9^2 +
  d \, I_6^2 I_{12}^2 +
  e \, I_6 I_9^2 I_{12} +
  f \, I_9^4 +
  g \, I_{12}^3,
  \]
for some $a, b, c, d, e, f, g \in \mathbb{C}$.
Determine these coefficients.
\end{openproblem}


\section*{Acknowledgements}

The first author was partially supported by a Discovery Grant from NSERC.
The authors thank the referees and Luke Oeding for helpful comments.


  \begin{table} \tiny
  \[
  \begin{array}{ccccc}
  I_{12} &\qquad I'_{12} &\qquad
  \text{MINIMAL REPRESENTATIVE} &\qquad \text{ORBIT SIZE} &\qquad \#
  \\
  \midrule
    0 &\qquad    1 &\qquad 0\,0\,0\,0\,0\,0\,0\,0\,4\,0\,0\,0\,0\,4\,0\,0\,0\,0\,4\,0\,0\,0\,0\,0\,0\,0\,0 &\qquad   36 &\qquad   1 \\
    0 &\qquad   -4 &\qquad 0\,0\,0\,0\,0\,0\,0\,0\,4\,0\,0\,0\,1\,3\,0\,0\,0\,0\,3\,1\,0\,0\,0\,0\,0\,0\,0 &\qquad  324 &\qquad   2 \\
    0 &\qquad    6 &\qquad 0\,0\,0\,0\,0\,0\,0\,0\,4\,0\,0\,0\,2\,2\,0\,0\,0\,0\,2\,2\,0\,0\,0\,0\,0\,0\,0 &\qquad  162 &\qquad   3 \\
    0 &\qquad    4 &\qquad 0\,0\,0\,0\,0\,0\,0\,0\,4\,0\,1\,0\,1\,2\,0\,0\,0\,0\,2\,1\,0\,1\,0\,0\,0\,0\,0 &\qquad  324 &\qquad   4 \\
    0 &\qquad    8 &\qquad 0\,0\,0\,0\,0\,0\,0\,0\,4\,0\,1\,0\,1\,2\,0\,0\,0\,0\,3\,0\,0\,0\,1\,0\,0\,0\,0 &\qquad  216 &\qquad   5 \\
    0 &\qquad  -16 &\qquad 0\,0\,0\,0\,0\,0\,0\,0\,4\,0\,1\,0\,2\,1\,0\,0\,0\,0\,2\,1\,0\,0\,1\,0\,0\,0\,0 &\qquad  324 &\qquad   6 \\
    0 &\qquad   16 &\qquad 0\,0\,0\,0\,0\,0\,0\,0\,4\,0\,2\,0\,2\,0\,0\,0\,0\,0\,2\,0\,0\,0\,2\,0\,0\,0\,0 &\qquad   54 &\qquad   7 \\
    0 &\qquad   56 &\qquad 0\,0\,0\,0\,0\,0\,0\,0\,4\,1\,1\,0\,1\,1\,0\,0\,0\,0\,1\,1\,0\,1\,1\,0\,0\,0\,0 &\qquad   27 &\qquad   8 \\
    0 &\qquad   12 &\qquad 0\,0\,0\,0\,0\,0\,0\,1\,3\,0\,0\,0\,1\,2\,1\,0\,0\,0\,3\,1\,0\,0\,0\,0\,0\,0\,0 &\qquad  324 &\qquad   9 \\
    0 &\qquad    4 &\qquad 0\,0\,0\,0\,0\,0\,0\,1\,3\,0\,0\,0\,1\,3\,0\,0\,0\,0\,3\,0\,1\,0\,0\,0\,0\,0\,0 &\qquad  216 &\qquad  10 \\
    0 &\qquad  -12 &\qquad 0\,0\,0\,0\,0\,0\,0\,1\,3\,0\,0\,0\,2\,1\,1\,0\,0\,0\,2\,2\,0\,0\,0\,0\,0\,0\,0 &\qquad  648 &\qquad  11 \\
    0 &\qquad    4 &\qquad 0\,0\,0\,0\,0\,0\,0\,1\,3\,0\,0\,1\,0\,3\,0\,0\,0\,0\,3\,0\,0\,1\,0\,0\,0\,0\,0 &\qquad  648 &\qquad  12 \\
    0 &\qquad   -4 &\qquad 0\,0\,0\,0\,0\,0\,0\,1\,3\,0\,0\,1\,1\,2\,0\,0\,0\,0\,2\,1\,0\,1\,0\,0\,0\,0\,0 &\qquad 1296 &\qquad  13 \\
    0 &\qquad   -8 &\qquad 0\,0\,0\,0\,0\,0\,0\,1\,3\,0\,0\,1\,1\,2\,0\,0\,0\,0\,3\,0\,0\,0\,1\,0\,0\,0\,0 &\qquad 1296 &\qquad  14 \\
    0 &\qquad   16 &\qquad 0\,0\,0\,0\,0\,0\,0\,1\,3\,0\,0\,1\,2\,1\,0\,0\,0\,0\,2\,1\,0\,0\,1\,0\,0\,0\,0 &\qquad  648 &\qquad  15 \\
    0 &\qquad   12 &\qquad 0\,0\,0\,0\,0\,0\,0\,1\,3\,0\,1\,0\,0\,2\,1\,0\,0\,0\,3\,0\,0\,1\,0\,0\,0\,0\,0 &\qquad  648 &\qquad  16 \\
    0 &\qquad   -8 &\qquad 0\,0\,0\,0\,0\,0\,0\,1\,3\,0\,1\,0\,1\,1\,1\,0\,0\,0\,2\,1\,0\,1\,0\,0\,0\,0\,0 &\qquad 1296 &\qquad  17 \\
    0 &\qquad  -16 &\qquad 0\,0\,0\,0\,0\,0\,0\,1\,3\,0\,1\,0\,1\,1\,1\,0\,0\,0\,3\,0\,0\,0\,1\,0\,0\,0\,0 &\qquad  648 &\qquad  18 \\
    0 &\qquad   -4 &\qquad 0\,0\,0\,0\,0\,0\,0\,1\,3\,0\,1\,0\,1\,2\,0\,0\,0\,0\,2\,0\,1\,1\,0\,0\,0\,0\,0 &\qquad 1296 &\qquad  19 \\
    0 &\qquad   16 &\qquad 0\,0\,0\,0\,0\,0\,0\,1\,3\,0\,1\,0\,2\,0\,1\,0\,0\,0\,2\,1\,0\,0\,1\,0\,0\,0\,0 &\qquad  648 &\qquad  20 \\
    0 &\qquad  -12 &\qquad 0\,0\,0\,0\,0\,0\,0\,1\,3\,0\,1\,1\,0\,2\,0\,0\,0\,0\,2\,0\,0\,2\,0\,0\,0\,0\,0 &\qquad 1296 &\qquad  21 \\
    0 &\qquad   -8 &\qquad 0\,0\,0\,0\,0\,0\,0\,1\,3\,0\,1\,1\,1\,1\,0\,0\,0\,0\,1\,1\,0\,2\,0\,0\,0\,0\,0 &\qquad  648 &\qquad  22 \\
    0 &\qquad   32 &\qquad 0\,0\,0\,0\,0\,0\,0\,1\,3\,0\,1\,1\,1\,1\,0\,0\,0\,0\,2\,0\,0\,1\,1\,0\,0\,0\,0 &\qquad 1296 &\qquad  23 \\
    0 &\qquad  -32 &\qquad 0\,0\,0\,0\,0\,0\,0\,1\,3\,0\,1\,1\,2\,0\,0\,0\,0\,0\,2\,0\,0\,0\,2\,0\,0\,0\,0 &\qquad  648 &\qquad  24 \\
    0 &\qquad   -4 &\qquad 0\,0\,0\,0\,0\,0\,0\,1\,3\,0\,2\,0\,1\,0\,1\,0\,0\,0\,1\,1\,0\,2\,0\,0\,0\,0\,0 &\qquad 1296 &\qquad  25 \\
    0 &\qquad   16 &\qquad 0\,0\,0\,0\,0\,0\,0\,1\,3\,0\,2\,0\,1\,0\,1\,0\,0\,0\,2\,0\,0\,1\,1\,0\,0\,0\,0 &\qquad 1296 &\qquad  26 \\
    0 &\qquad   16 &\qquad 0\,0\,0\,0\,0\,0\,0\,1\,3\,0\,2\,0\,1\,1\,0\,0\,0\,0\,2\,0\,0\,1\,0\,1\,0\,0\,0 &\qquad 1296 &\qquad  27 \\
    0 &\qquad  -24 &\qquad 0\,0\,0\,0\,0\,0\,0\,1\,3\,0\,2\,1\,1\,0\,0\,0\,0\,0\,1\,0\,0\,2\,1\,0\,0\,0\,0 &\qquad  648 &\qquad  28 \\
    0 &\qquad   -8 &\qquad 0\,0\,0\,0\,0\,0\,0\,1\,3\,0\,3\,0\,1\,0\,0\,0\,0\,0\,1\,0\,0\,2\,0\,1\,0\,0\,0 &\qquad  648 &\qquad  29 \\
    0 &\qquad   32 &\qquad 0\,0\,0\,0\,0\,0\,0\,1\,3\,1\,0\,0\,1\,1\,1\,0\,0\,0\,1\,2\,0\,1\,0\,0\,0\,0\,0 &\qquad  648 &\qquad  30 \\
    0 &\qquad  -56 &\qquad 0\,0\,0\,0\,0\,0\,0\,1\,3\,1\,0\,1\,1\,1\,0\,0\,0\,0\,1\,1\,0\,1\,1\,0\,0\,0\,0 &\qquad  648 &\qquad  31 \\
    0 &\qquad    6 &\qquad 0\,0\,0\,0\,0\,0\,0\,2\,2\,0\,0\,0\,2\,0\,2\,0\,0\,0\,2\,2\,0\,0\,0\,0\,0\,0\,0 &\qquad  108 &\qquad  32 \\
    0 &\qquad   24 &\qquad 0\,0\,0\,0\,0\,0\,0\,2\,2\,0\,0\,0\,2\,1\,1\,0\,0\,0\,2\,1\,1\,0\,0\,0\,0\,0\,0 &\qquad  162 &\qquad  33 \\
    0 &\qquad    8 &\qquad 0\,0\,0\,0\,0\,0\,0\,2\,2\,0\,0\,1\,1\,1\,1\,0\,0\,0\,2\,1\,0\,1\,0\,0\,0\,0\,0 &\qquad 1296 &\qquad  34 \\
    0 &\qquad   16 &\qquad 0\,0\,0\,0\,0\,0\,0\,2\,2\,0\,0\,1\,1\,1\,1\,0\,0\,0\,3\,0\,0\,0\,1\,0\,0\,0\,0 &\qquad  648 &\qquad  35 \\
    0 &\qquad    4 &\qquad 0\,0\,0\,0\,0\,0\,0\,2\,2\,0\,0\,1\,1\,2\,0\,0\,0\,0\,2\,0\,1\,1\,0\,0\,0\,0\,0 &\qquad 1296 &\qquad  36 \\
    0 &\qquad    8 &\qquad 0\,0\,0\,0\,0\,0\,0\,2\,2\,0\,0\,1\,1\,2\,0\,0\,0\,0\,3\,0\,0\,0\,0\,1\,0\,0\,0 &\qquad  648 &\qquad  37 \\
    0 &\qquad  -16 &\qquad 0\,0\,0\,0\,0\,0\,0\,2\,2\,0\,0\,1\,2\,0\,1\,0\,0\,0\,2\,1\,0\,0\,1\,0\,0\,0\,0 &\qquad  648 &\qquad  38 \\
    0 &\qquad  -16 &\qquad 0\,0\,0\,0\,0\,0\,0\,2\,2\,0\,0\,1\,2\,1\,0\,0\,0\,0\,2\,1\,0\,0\,0\,1\,0\,0\,0 &\qquad  324 &\qquad  39 \\
    0 &\qquad    6 &\qquad 0\,0\,0\,0\,0\,0\,0\,2\,2\,0\,0\,2\,0\,2\,0\,0\,0\,0\,2\,0\,0\,2\,0\,0\,0\,0\,0 &\qquad  324 &\qquad  40 \\
    0 &\qquad    4 &\qquad 0\,0\,0\,0\,0\,0\,0\,2\,2\,0\,0\,2\,1\,1\,0\,0\,0\,0\,1\,1\,0\,2\,0\,0\,0\,0\,0 &\qquad  648 &\qquad  41 \\
    0 &\qquad  -16 &\qquad 0\,0\,0\,0\,0\,0\,0\,2\,2\,0\,0\,2\,1\,1\,0\,0\,0\,0\,2\,0\,0\,1\,1\,0\,0\,0\,0 &\qquad 1296 &\qquad  42 \\
    0 &\qquad   16 &\qquad 0\,0\,0\,0\,0\,0\,0\,2\,2\,0\,0\,2\,2\,0\,0\,0\,0\,0\,2\,0\,0\,0\,2\,0\,0\,0\,0 &\qquad  324 &\qquad  43 \\
    0 &\qquad   24 &\qquad 0\,0\,0\,0\,0\,0\,0\,2\,2\,0\,1\,1\,0\,1\,1\,0\,0\,0\,2\,0\,0\,2\,0\,0\,0\,0\,0 &\qquad  162 &\qquad  44 \\
    0 &\qquad    8 &\qquad 0\,0\,0\,0\,0\,0\,0\,2\,2\,0\,1\,1\,1\,0\,1\,0\,0\,0\,1\,1\,0\,2\,0\,0\,0\,0\,0 &\qquad  648 &\qquad  45 \\
    0 &\qquad  -32 &\qquad 0\,0\,0\,0\,0\,0\,0\,2\,2\,0\,1\,1\,1\,0\,1\,0\,0\,0\,2\,0\,0\,1\,1\,0\,0\,0\,0 &\qquad 1296 &\qquad  46 \\
    0 &\qquad   64 &\qquad 0\,0\,0\,0\,0\,0\,0\,2\,2\,0\,1\,1\,2\,0\,0\,0\,0\,0\,2\,0\,0\,0\,1\,1\,0\,0\,0 &\qquad  162 &\qquad  47 \\
    0 &\qquad   24 &\qquad 0\,0\,0\,0\,0\,0\,0\,2\,2\,0\,1\,2\,1\,0\,0\,0\,0\,0\,1\,0\,0\,2\,1\,0\,0\,0\,0 &\qquad  648 &\qquad  48 \\
    0 &\qquad  -16 &\qquad 0\,0\,0\,0\,0\,0\,0\,2\,2\,1\,0\,0\,1\,0\,2\,0\,0\,0\,1\,2\,0\,1\,0\,0\,0\,0\,0 &\qquad  324 &\qquad  49 \\
    0 &\qquad  -64 &\qquad 0\,0\,0\,0\,0\,0\,0\,2\,2\,1\,0\,0\,1\,1\,1\,0\,0\,0\,1\,1\,1\,1\,0\,0\,0\,0\,0 &\qquad  162 &\qquad  50 \\
    0 &\qquad   56 &\qquad 0\,0\,0\,0\,0\,0\,0\,2\,2\,1\,0\,1\,1\,0\,1\,0\,0\,0\,1\,1\,0\,1\,1\,0\,0\,0\,0 &\qquad  324 &\qquad  51 \\
    0 &\qquad   56 &\qquad 0\,0\,0\,0\,0\,0\,0\,2\,2\,1\,0\,1\,1\,1\,0\,0\,0\,0\,1\,1\,0\,1\,0\,1\,0\,0\,0 &\qquad  162 &\qquad  52 \\
    0 &\qquad  -12 &\qquad 0\,0\,0\,0\,0\,0\,1\,1\,2\,0\,0\,0\,1\,2\,1\,0\,0\,0\,2\,1\,1\,0\,0\,0\,0\,0\,0 &\qquad  108 &\qquad  53 \\
    0 &\qquad   28 &\qquad 0\,0\,0\,0\,0\,0\,1\,1\,2\,0\,0\,1\,0\,2\,1\,0\,0\,0\,2\,1\,0\,1\,0\,0\,0\,0\,0 &\qquad 1296 &\qquad  54 \\
    0 &\qquad   -4 &\qquad 0\,0\,0\,0\,0\,0\,1\,1\,2\,0\,0\,1\,0\,2\,1\,0\,0\,0\,3\,0\,0\,0\,1\,0\,0\,0\,0 &\qquad 1296 &\qquad  55 \\
    0 &\qquad   -8 &\qquad 0\,0\,0\,0\,0\,0\,1\,1\,2\,0\,0\,1\,0\,3\,0\,0\,0\,0\,3\,0\,0\,0\,0\,1\,0\,0\,0 &\qquad  324 &\qquad  56 \\
    0 &\qquad  -24 &\qquad 0\,0\,0\,0\,0\,0\,1\,1\,2\,0\,0\,1\,1\,1\,1\,0\,0\,0\,1\,2\,0\,1\,0\,0\,0\,0\,0 &\qquad 1296 &\qquad  57 \\
    0 &\qquad    8 &\qquad 0\,0\,0\,0\,0\,0\,1\,1\,2\,0\,0\,1\,1\,2\,0\,0\,0\,0\,2\,1\,0\,0\,0\,1\,0\,0\,0 &\qquad  324 &\qquad  58 \\
    0 &\qquad  -16 &\qquad 0\,0\,0\,0\,0\,0\,1\,1\,2\,0\,0\,2\,0\,2\,0\,0\,0\,0\,1\,1\,0\,2\,0\,0\,0\,0\,0 &\qquad 1296 &\qquad  59 \\
    0 &\qquad    4 &\qquad 0\,0\,0\,0\,0\,0\,1\,1\,2\,0\,0\,2\,0\,2\,0\,0\,0\,0\,2\,0\,0\,1\,1\,0\,0\,0\,0 &\qquad  648 &\qquad  60 \\
    0 &\qquad   56 &\qquad 0\,0\,0\,0\,0\,0\,1\,1\,2\,0\,0\,2\,1\,1\,0\,0\,0\,0\,1\,1\,0\,1\,1\,0\,0\,0\,0 &\qquad  324 &\qquad  61 \\
    0 &\qquad    8 &\qquad 0\,0\,0\,0\,0\,0\,1\,1\,2\,0\,1\,0\,0\,1\,2\,0\,0\,0\,2\,1\,0\,1\,0\,0\,0\,0\,0 &\qquad 1296 &\qquad  62 \\
    0 &\qquad   16 &\qquad 0\,0\,0\,0\,0\,0\,1\,1\,2\,0\,1\,0\,0\,1\,2\,0\,0\,0\,3\,0\,0\,0\,1\,0\,0\,0\,0 &\qquad  648 &\qquad  63 \\
    0 &\qquad  -32 &\qquad 0\,0\,0\,0\,0\,0\,1\,1\,2\,0\,1\,0\,0\,2\,1\,0\,0\,0\,2\,0\,1\,1\,0\,0\,0\,0\,0 &\qquad  648 &\qquad  64 \\
    0 &\qquad  -32 &\qquad 0\,0\,0\,0\,0\,0\,1\,1\,2\,0\,1\,0\,1\,0\,2\,0\,0\,0\,2\,1\,0\,0\,1\,0\,0\,0\,0 &\qquad  648 &\qquad  65 \\
    0 &\qquad   32 &\qquad 0\,0\,0\,0\,0\,0\,1\,1\,2\,0\,1\,0\,1\,1\,1\,0\,0\,0\,1\,1\,1\,1\,0\,0\,0\,0\,0 &\qquad  324 &\qquad  66 \\
    0 &\qquad   16 &\qquad 0\,0\,0\,0\,0\,0\,1\,1\,2\,0\,1\,0\,1\,1\,1\,0\,0\,0\,2\,0\,1\,0\,1\,0\,0\,0\,0 &\qquad  648 &\qquad  67 \\
    0 &\qquad  -24 &\qquad 0\,0\,0\,0\,0\,0\,1\,1\,2\,0\,1\,1\,0\,1\,1\,0\,0\,0\,1\,1\,0\,2\,0\,0\,0\,0\,0 &\qquad 1296 &\qquad  68 \\
    0 &\qquad    8 &\qquad 0\,0\,0\,0\,0\,0\,1\,1\,2\,0\,1\,1\,0\,2\,0\,0\,0\,0\,1\,0\,1\,2\,0\,0\,0\,0\,0 &\qquad  648 &\qquad  69 \\
    0 &\qquad   28 &\qquad 0\,0\,0\,0\,0\,0\,1\,1\,2\,0\,1\,1\,0\,2\,0\,0\,0\,0\,2\,0\,0\,1\,0\,1\,0\,0\,0 &\qquad  648 &\qquad  70 \\
    0 &\qquad   -8 &\qquad 0\,0\,0\,0\,0\,0\,1\,1\,2\,0\,1\,1\,1\,0\,1\,0\,0\,0\,1\,1\,0\,1\,1\,0\,0\,0\,0 &\qquad  648 &\qquad  71 \\
    0 &\qquad   48 &\qquad 0\,0\,0\,0\,0\,0\,1\,1\,2\,0\,1\,1\,1\,0\,1\,0\,0\,0\,2\,0\,0\,0\,2\,0\,0\,0\,0 &\qquad  648 &\qquad  72 \\
  \bottomrule
  \end{array}
  \]
  \bigskip
  \caption{The basis invariants in degree 12: part 1}
  \label{degree12table1}
  \end{table}

  \begin{table} \tiny
  \[
  \begin{array}{ccccc}
  I_{12} &\qquad I'_{12} &\qquad
  \text{MINIMAL REPRESENTATIVE} &\qquad \text{ORBIT SIZE} &\qquad \#
  \\
  \midrule
    0 &\qquad   72 &\qquad 0\,0\,0\,0\,0\,0\,1\,1\,2\,0\,1\,1\,1\,1\,0\,0\,0\,0\,1\,1\,0\,1\,0\,1\,0\,0\,0 &\qquad  324 &\qquad  73 \\
    0 &\qquad  -64 &\qquad 0\,0\,0\,0\,0\,0\,1\,1\,2\,0\,1\,1\,1\,1\,0\,0\,0\,0\,2\,0\,0\,0\,1\,1\,0\,0\,0 &\qquad  648 &\qquad  74 \\
    0 &\qquad  -12 &\qquad 0\,0\,0\,0\,0\,1\,0\,1\,2\,0\,0\,1\,0\,2\,0\,1\,0\,0\,2\,1\,0\,1\,0\,0\,0\,0\,0 &\qquad  324 &\qquad  75 \\
    0 &\qquad   16 &\qquad 0\,0\,0\,0\,0\,1\,0\,1\,2\,0\,0\,1\,0\,2\,0\,1\,0\,0\,3\,0\,0\,0\,1\,0\,0\,0\,0 &\qquad  648 &\qquad  76 \\
    1 &\qquad   13 &\qquad 0\,0\,0\,0\,0\,1\,0\,1\,2\,0\,0\,1\,1\,1\,0\,0\,1\,0\,2\,1\,0\,1\,0\,0\,0\,0\,0 &\qquad  648 &\qquad  77 \\
   -1 &\qquad   -5 &\qquad 0\,0\,0\,0\,0\,1\,0\,1\,2\,0\,0\,1\,1\,1\,0\,0\,1\,0\,3\,0\,0\,0\,1\,0\,0\,0\,0 &\qquad 1296 &\qquad  78 \\
   -1 &\qquad   19 &\qquad 0\,0\,0\,0\,0\,1\,0\,1\,2\,0\,0\,1\,1\,1\,0\,1\,0\,0\,1\,2\,0\,1\,0\,0\,0\,0\,0 &\qquad 1296 &\qquad  79 \\
    1 &\qquad  -27 &\qquad 0\,0\,0\,0\,0\,1\,0\,1\,2\,0\,0\,1\,1\,1\,0\,1\,0\,0\,2\,1\,0\,0\,1\,0\,0\,0\,0 &\qquad 1296 &\qquad  80 \\
   -1 &\qquad   11 &\qquad 0\,0\,0\,0\,0\,1\,0\,1\,2\,0\,0\,1\,1\,2\,0\,0\,0\,0\,2\,1\,0\,0\,0\,0\,1\,0\,0 &\qquad 1296 &\qquad  81 \\
    1 &\qquad   13 &\qquad 0\,0\,0\,0\,0\,1\,0\,1\,2\,0\,0\,1\,1\,2\,0\,0\,0\,0\,3\,0\,0\,0\,0\,0\,0\,1\,0 &\qquad 1296 &\qquad  82 \\
   -1 &\qquad   -9 &\qquad 0\,0\,0\,0\,0\,1\,0\,1\,2\,0\,0\,1\,2\,0\,0\,0\,1\,0\,1\,2\,0\,1\,0\,0\,0\,0\,0 &\qquad 1296 &\qquad  83 \\
    1 &\qquad    5 &\qquad 0\,0\,0\,0\,0\,1\,0\,1\,2\,0\,0\,1\,2\,0\,0\,0\,1\,0\,2\,1\,0\,0\,1\,0\,0\,0\,0 &\qquad 1296 &\qquad  84 \\
    1 &\qquad   -7 &\qquad 0\,0\,0\,0\,0\,1\,0\,1\,2\,0\,0\,1\,2\,0\,0\,1\,0\,0\,0\,3\,0\,1\,0\,0\,0\,0\,0 &\qquad  648 &\qquad  85 \\
    2 &\qquad    2 &\qquad 0\,0\,0\,0\,0\,1\,0\,1\,2\,0\,0\,1\,2\,1\,0\,0\,0\,0\,1\,2\,0\,0\,0\,0\,1\,0\,0 &\qquad 1296 &\qquad  86 \\
   -2 &\qquad  -26 &\qquad 0\,0\,0\,0\,0\,1\,0\,1\,2\,0\,0\,1\,2\,1\,0\,0\,0\,0\,2\,1\,0\,0\,0\,0\,0\,1\,0 &\qquad 1296 &\qquad  87 \\
   -1 &\qquad   -5 &\qquad 0\,0\,0\,0\,0\,1\,0\,1\,2\,0\,0\,1\,3\,0\,0\,0\,0\,0\,0\,3\,0\,0\,0\,0\,1\,0\,0 &\qquad  648 &\qquad  88 \\
    1 &\qquad   13 &\qquad 0\,0\,0\,0\,0\,1\,0\,1\,2\,0\,0\,1\,3\,0\,0\,0\,0\,0\,1\,2\,0\,0\,0\,0\,0\,1\,0 &\qquad  648 &\qquad  89 \\
   -2 &\qquad  -18 &\qquad 0\,0\,0\,0\,0\,1\,0\,1\,2\,0\,1\,0\,1\,0\,1\,0\,1\,0\,2\,1\,0\,1\,0\,0\,0\,0\,0 &\qquad  324 &\qquad  90 \\
    2 &\qquad   26 &\qquad 0\,0\,0\,0\,0\,1\,0\,1\,2\,0\,1\,0\,1\,0\,1\,0\,1\,0\,3\,0\,0\,0\,1\,0\,0\,0\,0 &\qquad  648 &\qquad  91 \\
    1 &\qquad   13 &\qquad 0\,0\,0\,0\,0\,1\,0\,1\,2\,0\,1\,0\,1\,0\,1\,1\,0\,0\,1\,2\,0\,1\,0\,0\,0\,0\,0 &\qquad 1296 &\qquad  92 \\
   -1 &\qquad  -21 &\qquad 0\,0\,0\,0\,0\,1\,0\,1\,2\,0\,1\,0\,1\,0\,1\,1\,0\,0\,2\,1\,0\,0\,1\,0\,0\,0\,0 &\qquad 1296 &\qquad  93 \\
    0 &\qquad   16 &\qquad 0\,0\,0\,0\,0\,1\,0\,1\,2\,0\,1\,0\,1\,1\,0\,0\,0\,1\,2\,1\,0\,1\,0\,0\,0\,0\,0 &\qquad  648 &\qquad  94 \\
    0 &\qquad   32 &\qquad 0\,0\,0\,0\,0\,1\,0\,1\,2\,0\,1\,0\,1\,1\,0\,0\,0\,1\,3\,0\,0\,0\,1\,0\,0\,0\,0 &\qquad  216 &\qquad  95 \\
    1 &\qquad   13 &\qquad 0\,0\,0\,0\,0\,1\,0\,1\,2\,0\,1\,0\,1\,1\,0\,0\,1\,0\,2\,0\,1\,1\,0\,0\,0\,0\,0 &\qquad 1296 &\qquad  96 \\
   -1 &\qquad   -5 &\qquad 0\,0\,0\,0\,0\,1\,0\,1\,2\,0\,1\,0\,1\,1\,0\,0\,1\,0\,3\,0\,0\,0\,0\,1\,0\,0\,0 &\qquad 1296 &\qquad  97 \\
    0 &\qquad   -8 &\qquad 0\,0\,0\,0\,0\,1\,0\,1\,2\,0\,1\,0\,1\,1\,0\,1\,0\,0\,1\,1\,1\,1\,0\,0\,0\,0\,0 &\qquad 1296 &\qquad  98 \\
    1 &\qquad  -27 &\qquad 0\,0\,0\,0\,0\,1\,0\,1\,2\,0\,1\,0\,1\,1\,0\,1\,0\,0\,2\,0\,1\,0\,1\,0\,0\,0\,0 &\qquad 1296 &\qquad  99 \\
   -1 &\qquad  -21 &\qquad 0\,0\,0\,0\,0\,1\,0\,1\,2\,0\,1\,0\,1\,1\,0\,1\,0\,0\,2\,1\,0\,0\,0\,1\,0\,0\,0 &\qquad 1296 &\qquad 100 \\
    0 &\qquad  -12 &\qquad 0\,0\,0\,0\,0\,1\,0\,1\,2\,0\,1\,0\,1\,2\,0\,0\,0\,0\,2\,0\,1\,0\,0\,0\,1\,0\,0 &\qquad  216 &\qquad 101 \\
    0 &\qquad   16 &\qquad 0\,0\,0\,0\,0\,1\,0\,1\,2\,0\,1\,0\,1\,2\,0\,0\,0\,0\,3\,0\,0\,0\,0\,0\,0\,0\,1 &\qquad  648 &\qquad 102 \\
    0 &\qquad    8 &\qquad 0\,0\,0\,0\,0\,1\,0\,1\,2\,0\,1\,0\,2\,0\,0\,0\,0\,1\,1\,2\,0\,1\,0\,0\,0\,0\,0 &\qquad 1296 &\qquad 103 \\
    0 &\qquad  -32 &\qquad 0\,0\,0\,0\,0\,1\,0\,1\,2\,0\,1\,0\,2\,0\,0\,0\,0\,1\,2\,1\,0\,0\,1\,0\,0\,0\,0 &\qquad  648 &\qquad 104 \\
    1 &\qquad   13 &\qquad 0\,0\,0\,0\,0\,1\,0\,1\,2\,0\,1\,0\,2\,0\,0\,0\,1\,0\,1\,1\,1\,1\,0\,0\,0\,0\,0 &\qquad 1296 &\qquad 105 \\
   -2 &\qquad  -26 &\qquad 0\,0\,0\,0\,0\,1\,0\,1\,2\,0\,1\,0\,2\,0\,0\,0\,1\,0\,2\,0\,1\,0\,1\,0\,0\,0\,0 &\qquad 1296 &\qquad 106 \\
    1 &\qquad    5 &\qquad 0\,0\,0\,0\,0\,1\,0\,1\,2\,0\,1\,0\,2\,0\,0\,0\,1\,0\,2\,1\,0\,0\,0\,1\,0\,0\,0 &\qquad 1296 &\qquad 107 \\
   -1 &\qquad  -29 &\qquad 0\,0\,0\,0\,0\,1\,0\,1\,2\,0\,1\,0\,2\,0\,0\,1\,0\,0\,0\,2\,1\,1\,0\,0\,0\,0\,0 &\qquad  216 &\qquad 108 \\
    1 &\qquad   53 &\qquad 0\,0\,0\,0\,0\,1\,0\,1\,2\,0\,1\,0\,2\,0\,0\,1\,0\,0\,1\,1\,1\,0\,1\,0\,0\,0\,0 &\qquad 1296 &\qquad 109 \\
    0 &\qquad    8 &\qquad 0\,0\,0\,0\,0\,1\,0\,1\,2\,0\,1\,0\,2\,0\,0\,1\,0\,0\,1\,2\,0\,0\,0\,1\,0\,0\,0 &\qquad 1296 &\qquad 110 \\
   -1 &\qquad   -9 &\qquad 0\,0\,0\,0\,0\,1\,0\,1\,2\,0\,1\,0\,2\,0\,1\,0\,0\,0\,1\,2\,0\,0\,0\,0\,1\,0\,0 &\qquad 1296 &\qquad 111 \\
    1 &\qquad    5 &\qquad 0\,0\,0\,0\,0\,1\,0\,1\,2\,0\,1\,0\,2\,0\,1\,0\,0\,0\,2\,1\,0\,0\,0\,0\,0\,1\,0 &\qquad 1296 &\qquad 112 \\
   -1 &\qquad   19 &\qquad 0\,0\,0\,0\,0\,1\,0\,1\,2\,0\,1\,0\,2\,1\,0\,0\,0\,0\,1\,1\,1\,0\,0\,0\,1\,0\,0 &\qquad 1296 &\qquad 113 \\
    1 &\qquad    5 &\qquad 0\,0\,0\,0\,0\,1\,0\,1\,2\,0\,1\,0\,2\,1\,0\,0\,0\,0\,2\,0\,1\,0\,0\,0\,0\,1\,0 &\qquad 1296 &\qquad 114 \\
    0 &\qquad  -32 &\qquad 0\,0\,0\,0\,0\,1\,0\,1\,2\,0\,1\,0\,2\,1\,0\,0\,0\,0\,2\,1\,0\,0\,0\,0\,0\,0\,1 &\qquad  648 &\qquad 115 \\
   -1 &\qquad   -5 &\qquad 0\,0\,0\,0\,0\,1\,0\,1\,2\,0\,1\,0\,3\,0\,0\,0\,0\,0\,1\,1\,1\,0\,0\,0\,0\,1\,0 &\qquad 1296 &\qquad 116 \\
   -2 &\qquad   14 &\qquad 0\,0\,0\,0\,0\,1\,0\,1\,2\,0\,1\,1\,1\,0\,0\,1\,0\,0\,0\,2\,0\,2\,0\,0\,0\,0\,0 &\qquad  648 &\qquad 117 \\
    4 &\qquad  -60 &\qquad 0\,0\,0\,0\,0\,1\,0\,1\,2\,0\,1\,1\,1\,0\,0\,1\,0\,0\,1\,1\,0\,1\,1\,0\,0\,0\,0 &\qquad  648 &\qquad 118 \\
   -2 &\qquad   54 &\qquad 0\,0\,0\,0\,0\,1\,0\,1\,2\,0\,1\,1\,1\,0\,0\,1\,0\,0\,2\,0\,0\,0\,2\,0\,0\,0\,0 &\qquad  648 &\qquad 119 \\
    2 &\qquad   22 &\qquad 0\,0\,0\,0\,0\,1\,0\,1\,2\,0\,1\,1\,2\,0\,0\,0\,0\,0\,0\,2\,0\,1\,0\,0\,1\,0\,0 &\qquad  648 &\qquad 120 \\
   -3 &\qquad  -15 &\qquad 0\,0\,0\,0\,0\,1\,0\,1\,2\,0\,1\,1\,2\,0\,0\,0\,0\,0\,1\,1\,0\,0\,1\,0\,1\,0\,0 &\qquad 1296 &\qquad 121 \\
   -1 &\qquad  -21 &\qquad 0\,0\,0\,0\,0\,1\,0\,1\,2\,0\,1\,1\,2\,0\,0\,0\,0\,0\,1\,1\,0\,1\,0\,0\,0\,1\,0 &\qquad 1296 &\qquad 122 \\
    2 &\qquad   42 &\qquad 0\,0\,0\,0\,0\,1\,0\,1\,2\,0\,1\,1\,2\,0\,0\,0\,0\,0\,2\,0\,0\,0\,1\,0\,0\,1\,0 &\qquad 1296 &\qquad 123 \\
    2 &\qquad  -38 &\qquad 0\,0\,0\,0\,0\,1\,0\,1\,2\,0\,2\,0\,2\,0\,0\,0\,0\,0\,1\,0\,1\,0\,1\,0\,1\,0\,0 &\qquad  648 &\qquad 124 \\
    1 &\qquad    5 &\qquad 0\,0\,0\,0\,0\,1\,0\,1\,2\,0\,2\,0\,2\,0\,0\,0\,0\,0\,1\,0\,1\,1\,0\,0\,0\,1\,0 &\qquad 1296 &\qquad 125 \\
   -2 &\qquad  -10 &\qquad 0\,0\,0\,0\,0\,1\,0\,1\,2\,0\,2\,0\,2\,0\,0\,0\,0\,0\,2\,0\,0\,0\,0\,1\,0\,1\,0 &\qquad  648 &\qquad 126 \\
    0 &\qquad   64 &\qquad 0\,0\,0\,0\,0\,1\,0\,1\,2\,0\,2\,0\,2\,0\,0\,0\,0\,0\,2\,0\,0\,0\,1\,0\,0\,0\,1 &\qquad  216 &\qquad 127 \\
   -1 &\qquad  -21 &\qquad 0\,0\,0\,0\,0\,1\,0\,1\,2\,1\,0\,0\,0\,1\,1\,1\,0\,0\,1\,2\,0\,1\,0\,0\,0\,0\,0 &\qquad 1296 &\qquad 128 \\
    1 &\qquad   13 &\qquad 0\,0\,0\,0\,0\,1\,0\,1\,2\,1\,0\,0\,0\,1\,1\,1\,0\,0\,2\,1\,0\,0\,1\,0\,0\,0\,0 &\qquad 1296 &\qquad 129 \\
    0 &\qquad   48 &\qquad 0\,0\,0\,0\,0\,1\,0\,1\,2\,1\,0\,0\,0\,2\,0\,0\,0\,1\,2\,1\,0\,1\,0\,0\,0\,0\,0 &\qquad  108 &\qquad 130 \\
    1 &\qquad  -27 &\qquad 0\,0\,0\,0\,0\,1\,0\,1\,2\,1\,0\,0\,0\,2\,0\,1\,0\,0\,1\,1\,1\,1\,0\,0\,0\,0\,0 &\qquad 1296 &\qquad 131 \\
   -1 &\qquad   -9 &\qquad 0\,0\,0\,0\,0\,1\,0\,1\,2\,1\,0\,0\,0\,2\,0\,1\,0\,0\,2\,1\,0\,0\,0\,1\,0\,0\,0 &\qquad  648 &\qquad 132 \\
    2 &\qquad   10 &\qquad 0\,0\,0\,0\,0\,1\,0\,1\,2\,1\,0\,0\,1\,0\,1\,0\,1\,0\,1\,2\,0\,1\,0\,0\,0\,0\,0 &\qquad 1296 &\qquad 133 \\
   -2 &\qquad  -18 &\qquad 0\,0\,0\,0\,0\,1\,0\,1\,2\,1\,0\,0\,1\,0\,1\,0\,1\,0\,2\,1\,0\,0\,1\,0\,0\,0\,0 &\qquad  648 &\qquad 134 \\
    0 &\qquad  -64 &\qquad 0\,0\,0\,0\,0\,1\,0\,1\,2\,1\,0\,0\,1\,1\,0\,0\,0\,1\,1\,2\,0\,1\,0\,0\,0\,0\,0 &\qquad  648 &\qquad 135 \\
   -2 &\qquad  -42 &\qquad 0\,0\,0\,0\,0\,1\,0\,1\,2\,1\,0\,0\,1\,1\,0\,0\,1\,0\,1\,1\,1\,1\,0\,0\,0\,0\,0 &\qquad 1296 &\qquad 136 \\
   -2 &\qquad   -2 &\qquad 0\,0\,0\,0\,0\,1\,0\,1\,2\,1\,0\,0\,1\,1\,0\,1\,0\,0\,1\,1\,1\,0\,1\,0\,0\,0\,0 &\qquad  648 &\qquad 137 \\
    1 &\qquad   13 &\qquad 0\,0\,0\,0\,0\,1\,0\,1\,2\,1\,0\,0\,1\,1\,0\,1\,0\,0\,1\,2\,0\,0\,0\,1\,0\,0\,0 &\qquad 1296 &\qquad 138 \\
   -1 &\qquad  -21 &\qquad 0\,0\,0\,0\,0\,1\,0\,1\,2\,1\,0\,0\,1\,1\,1\,0\,0\,0\,1\,2\,0\,0\,0\,0\,1\,0\,0 &\qquad 1296 &\qquad 139 \\
    1 &\qquad   13 &\qquad 0\,0\,0\,0\,0\,1\,0\,1\,2\,1\,0\,0\,1\,1\,1\,0\,0\,0\,2\,1\,0\,0\,0\,0\,0\,1\,0 &\qquad  648 &\qquad 140 \\
    1 &\qquad  -27 &\qquad 0\,0\,0\,0\,0\,1\,0\,1\,2\,1\,0\,0\,1\,2\,0\,0\,0\,0\,1\,1\,1\,0\,0\,0\,1\,0\,0 &\qquad 1296 &\qquad 141 \\
   -1 &\qquad   -9 &\qquad 0\,0\,0\,0\,0\,1\,0\,1\,2\,1\,0\,0\,2\,0\,0\,0\,1\,0\,1\,2\,0\,0\,0\,1\,0\,0\,0 &\qquad  648 &\qquad 142 \\
   -1 &\qquad   -9 &\qquad 0\,0\,0\,0\,0\,1\,0\,1\,2\,1\,0\,0\,2\,0\,1\,0\,0\,0\,1\,2\,0\,0\,0\,0\,0\,1\,0 &\qquad  216 &\qquad 143 \\
   -2 &\qquad   78 &\qquad 0\,0\,0\,0\,0\,1\,0\,1\,2\,1\,0\,1\,0\,1\,0\,1\,0\,0\,1\,1\,0\,1\,1\,0\,0\,0\,0 &\qquad  648 &\qquad 144 \\
  \bottomrule
  \end{array}
  \]
  \bigskip
  \caption{The basis invariants in degree 12: part 2}
  \label{degree12table2}
  \end{table}

  \begin{table} \tiny
  \[
  \begin{array}{ccccc}
  I_{12} &\qquad I'_{12} &\qquad
  \text{MINIMAL REPRESENTATIVE} &\qquad \text{ORBIT SIZE} &\qquad \#
  \\
  \midrule
    0 &\qquad  -32 &\qquad 0\,0\,0\,0\,0\,1\,0\,1\,2\,1\,0\,1\,0\,1\,0\,1\,0\,0\,2\,0\,0\,0\,2\,0\,0\,0\,0 &\qquad  648 &\qquad 145 \\
   -2 &\qquad   -2 &\qquad 0\,0\,0\,0\,0\,1\,0\,1\,2\,1\,0\,1\,1\,0\,0\,0\,1\,0\,1\,1\,0\,1\,1\,0\,0\,0\,0 &\qquad 1296 &\qquad 146 \\
    1 &\qquad    5 &\qquad 0\,0\,0\,0\,0\,1\,0\,1\,2\,1\,0\,1\,1\,0\,0\,0\,1\,0\,2\,0\,0\,0\,2\,0\,0\,0\,0 &\qquad 1296 &\qquad 147 \\
    0 &\qquad   32 &\qquad 0\,0\,0\,0\,0\,1\,0\,1\,2\,1\,0\,1\,1\,1\,0\,0\,0\,0\,1\,1\,0\,0\,1\,0\,1\,0\,0 &\qquad  648 &\qquad 148 \\
    4 &\qquad   60 &\qquad 0\,0\,0\,0\,0\,1\,0\,1\,2\,1\,0\,1\,1\,1\,0\,0\,0\,0\,1\,1\,0\,1\,0\,0\,0\,1\,0 &\qquad 1296 &\qquad 149 \\
   -1 &\qquad  -21 &\qquad 0\,0\,0\,0\,0\,1\,0\,1\,2\,1\,0\,1\,1\,1\,0\,0\,0\,0\,2\,0\,0\,0\,1\,0\,0\,1\,0 &\qquad 1296 &\qquad 150 \\
    0 &\qquad  112 &\qquad 0\,0\,0\,0\,0\,1\,0\,1\,2\,1\,1\,0\,1\,0\,0\,0\,0\,1\,1\,1\,0\,1\,1\,0\,0\,0\,0 &\qquad  216 &\qquad 151 \\
   -2 &\qquad   -2 &\qquad 0\,0\,0\,0\,0\,1\,0\,1\,2\,1\,1\,0\,1\,0\,0\,0\,1\,0\,1\,1\,0\,1\,0\,1\,0\,0\,0 &\qquad  648 &\qquad 152 \\
   -1 &\qquad  -21 &\qquad 0\,0\,0\,0\,0\,1\,0\,1\,2\,1\,1\,0\,1\,0\,0\,0\,1\,0\,2\,0\,0\,0\,1\,1\,0\,0\,0 &\qquad 1296 &\qquad 153 \\
    2 &\qquad   10 &\qquad 0\,0\,0\,0\,0\,1\,0\,1\,2\,1\,1\,0\,1\,1\,0\,0\,0\,0\,2\,0\,0\,0\,0\,1\,0\,1\,0 &\qquad  648 &\qquad 154 \\
    1 &\qquad  -11 &\qquad 0\,0\,0\,0\,0\,1\,0\,2\,1\,0\,0\,1\,0\,1\,1\,1\,0\,0\,3\,0\,0\,0\,1\,0\,0\,0\,0 &\qquad 1296 &\qquad 155 \\
   -1 &\qquad  -21 &\qquad 0\,0\,0\,0\,0\,1\,0\,2\,1\,0\,0\,1\,0\,2\,0\,1\,0\,0\,3\,0\,0\,0\,0\,1\,0\,0\,0 &\qquad 1296 &\qquad 156 \\
    0 &\qquad  -24 &\qquad 0\,0\,0\,0\,0\,1\,0\,2\,1\,0\,0\,1\,1\,0\,1\,1\,0\,0\,1\,2\,0\,1\,0\,0\,0\,0\,0 &\qquad  648 &\qquad 157 \\
    0 &\qquad   16 &\qquad 0\,0\,0\,0\,0\,1\,0\,2\,1\,0\,0\,1\,1\,0\,1\,1\,0\,0\,2\,1\,0\,0\,1\,0\,0\,0\,0 &\qquad 1296 &\qquad 158 \\
    1 &\qquad  -11 &\qquad 0\,0\,0\,0\,0\,1\,0\,2\,1\,0\,0\,1\,1\,1\,0\,0\,0\,1\,3\,0\,0\,0\,1\,0\,0\,0\,0 &\qquad 1296 &\qquad 159 \\
    2 &\qquad  -14 &\qquad 0\,0\,0\,0\,0\,1\,0\,2\,1\,0\,0\,1\,1\,1\,0\,1\,0\,0\,1\,1\,1\,1\,0\,0\,0\,0\,0 &\qquad 1296 &\qquad 160 \\
   -2 &\qquad   22 &\qquad 0\,0\,0\,0\,0\,1\,0\,2\,1\,0\,0\,1\,1\,1\,0\,1\,0\,0\,2\,0\,1\,0\,1\,0\,0\,0\,0 &\qquad  648 &\qquad 161 \\
    0 &\qquad   16 &\qquad 0\,0\,0\,0\,0\,1\,0\,2\,1\,0\,0\,1\,1\,1\,0\,1\,0\,0\,2\,1\,0\,0\,0\,1\,0\,0\,0 &\qquad 1296 &\qquad 162 \\
   -1 &\qquad  -21 &\qquad 0\,0\,0\,0\,0\,1\,0\,2\,1\,0\,0\,1\,1\,2\,0\,0\,0\,0\,3\,0\,0\,0\,0\,0\,0\,0\,1 &\qquad 1296 &\qquad 163 \\
    1 &\qquad    5 &\qquad 0\,0\,0\,0\,0\,1\,0\,2\,1\,0\,0\,1\,2\,0\,0\,0\,0\,1\,1\,2\,0\,1\,0\,0\,0\,0\,0 &\qquad 1296 &\qquad 164 \\
   -1 &\qquad   11 &\qquad 0\,0\,0\,0\,0\,1\,0\,2\,1\,0\,0\,1\,2\,0\,0\,0\,0\,1\,2\,1\,0\,0\,1\,0\,0\,0\,0 &\qquad 1296 &\qquad 165 \\
    1 &\qquad  -27 &\qquad 0\,0\,0\,0\,0\,1\,0\,2\,1\,0\,0\,1\,2\,0\,0\,1\,0\,0\,1\,1\,1\,0\,1\,0\,0\,0\,0 &\qquad 1296 &\qquad 166 \\
    1 &\qquad    5 &\qquad 0\,0\,0\,0\,0\,1\,0\,2\,1\,0\,0\,1\,2\,0\,0\,1\,0\,0\,1\,2\,0\,0\,0\,1\,0\,0\,0 &\qquad 1296 &\qquad 167 \\
   -1 &\qquad   11 &\qquad 0\,0\,0\,0\,0\,1\,0\,2\,1\,0\,0\,1\,2\,0\,1\,0\,0\,0\,1\,2\,0\,0\,0\,0\,1\,0\,0 &\qquad 1296 &\qquad 168 \\
   -3 &\qquad  -15 &\qquad 0\,0\,0\,0\,0\,1\,0\,2\,1\,0\,0\,1\,2\,1\,0\,0\,0\,0\,1\,1\,1\,0\,0\,0\,1\,0\,0 &\qquad 1296 &\qquad 169 \\
    2 &\qquad   42 &\qquad 0\,0\,0\,0\,0\,1\,0\,2\,1\,0\,0\,1\,2\,1\,0\,0\,0\,0\,2\,1\,0\,0\,0\,0\,0\,0\,1 &\qquad  648 &\qquad 170 \\
    2 &\qquad   10 &\qquad 0\,0\,0\,0\,0\,1\,0\,2\,1\,0\,0\,1\,3\,0\,0\,0\,0\,0\,0\,2\,1\,0\,0\,0\,1\,0\,0 &\qquad  648 &\qquad 171 \\
    1 &\qquad  -15 &\qquad 0\,0\,0\,0\,0\,1\,0\,2\,1\,0\,0\,2\,0\,1\,0\,1\,0\,0\,1\,1\,0\,2\,0\,0\,0\,0\,0 &\qquad  648 &\qquad 172 \\
   -1 &\qquad   11 &\qquad 0\,0\,0\,0\,0\,1\,0\,2\,1\,0\,0\,2\,0\,1\,0\,1\,0\,0\,2\,0\,0\,1\,1\,0\,0\,0\,0 &\qquad 1296 &\qquad 173 \\
    2 &\qquad   -2 &\qquad 0\,0\,0\,0\,0\,1\,0\,2\,1\,0\,0\,2\,1\,0\,0\,1\,0\,0\,0\,2\,0\,2\,0\,0\,0\,0\,0 &\qquad  648 &\qquad 174 \\
   -4 &\qquad   28 &\qquad 0\,0\,0\,0\,0\,1\,0\,2\,1\,0\,0\,2\,1\,0\,0\,1\,0\,0\,1\,1\,0\,1\,1\,0\,0\,0\,0 &\qquad 1296 &\qquad 175 \\
    2 &\qquad  -22 &\qquad 0\,0\,0\,0\,0\,1\,0\,2\,1\,0\,0\,2\,1\,0\,0\,1\,0\,0\,2\,0\,0\,0\,2\,0\,0\,0\,0 &\qquad 1296 &\qquad 176 \\
    0 &\qquad   16 &\qquad 0\,0\,0\,0\,0\,1\,0\,2\,1\,0\,0\,2\,1\,1\,0\,0\,0\,0\,1\,1\,0\,1\,0\,0\,1\,0\,0 &\qquad 1296 &\qquad 177 \\
   -1 &\qquad   11 &\qquad 0\,0\,0\,0\,0\,1\,0\,2\,1\,0\,0\,2\,1\,1\,0\,0\,0\,0\,2\,0\,0\,0\,1\,0\,1\,0\,0 &\qquad 1296 &\qquad 178 \\
   -2 &\qquad  -10 &\qquad 0\,0\,0\,0\,0\,1\,0\,2\,1\,0\,0\,2\,2\,0\,0\,0\,0\,0\,0\,2\,0\,1\,0\,0\,1\,0\,0 &\qquad  648 &\qquad 179 \\
    3 &\qquad   -1 &\qquad 0\,0\,0\,0\,0\,1\,0\,2\,1\,0\,0\,2\,2\,0\,0\,0\,0\,0\,1\,1\,0\,0\,1\,0\,1\,0\,0 &\qquad 1296 &\qquad 180 \\
    1 &\qquad   13 &\qquad 0\,0\,0\,0\,0\,1\,0\,2\,1\,0\,1\,0\,0\,0\,2\,1\,0\,0\,3\,0\,0\,0\,1\,0\,0\,0\,0 &\qquad  648 &\qquad 181 \\
   -1 &\qquad   -5 &\qquad 0\,0\,0\,0\,0\,1\,0\,2\,1\,0\,1\,0\,0\,1\,1\,1\,0\,0\,3\,0\,0\,0\,0\,1\,0\,0\,0 &\qquad 1296 &\qquad 182 \\
   -2 &\qquad  -42 &\qquad 0\,0\,0\,0\,0\,1\,0\,2\,1\,0\,1\,0\,1\,0\,1\,0\,0\,1\,3\,0\,0\,0\,1\,0\,0\,0\,0 &\qquad  648 &\qquad 183 \\
   -2 &\qquad   -2 &\qquad 0\,0\,0\,0\,0\,1\,0\,2\,1\,0\,1\,0\,1\,0\,1\,1\,0\,0\,1\,1\,1\,1\,0\,0\,0\,0\,0 &\qquad 1296 &\qquad 184 \\
    0 &\qquad   16 &\qquad 0\,0\,0\,0\,0\,1\,0\,2\,1\,0\,1\,0\,1\,0\,1\,1\,0\,0\,2\,0\,1\,0\,1\,0\,0\,0\,0 &\qquad 1296 &\qquad 185 \\
    2 &\qquad   10 &\qquad 0\,0\,0\,0\,0\,1\,0\,2\,1\,0\,1\,0\,1\,0\,1\,1\,0\,0\,2\,1\,0\,0\,0\,1\,0\,0\,0 &\qquad 1296 &\qquad 186 \\
    0 &\qquad   16 &\qquad 0\,0\,0\,0\,0\,1\,0\,2\,1\,0\,1\,0\,1\,1\,0\,1\,0\,0\,2\,0\,1\,0\,0\,1\,0\,0\,0 &\qquad  648 &\qquad 187 \\
    1 &\qquad  -11 &\qquad 0\,0\,0\,0\,0\,1\,0\,2\,1\,0\,1\,0\,1\,1\,1\,0\,0\,0\,3\,0\,0\,0\,0\,0\,0\,0\,1 &\qquad 1296 &\qquad 188 \\
   -1 &\qquad  -21 &\qquad 0\,0\,0\,0\,0\,1\,0\,2\,1\,0\,1\,0\,2\,0\,0\,0\,0\,1\,1\,1\,1\,1\,0\,0\,0\,0\,0 &\qquad 1296 &\qquad 189 \\
    2 &\qquad   42 &\qquad 0\,0\,0\,0\,0\,1\,0\,2\,1\,0\,1\,0\,2\,0\,0\,0\,0\,1\,2\,0\,1\,0\,1\,0\,0\,0\,0 &\qquad  648 &\qquad 190 \\
   -1 &\qquad  -29 &\qquad 0\,0\,0\,0\,0\,1\,0\,2\,1\,0\,1\,0\,2\,0\,0\,1\,0\,0\,1\,0\,2\,0\,1\,0\,0\,0\,0 &\qquad 1296 &\qquad 191 \\
   -1 &\qquad  -21 &\qquad 0\,0\,0\,0\,0\,1\,0\,2\,1\,0\,1\,0\,2\,0\,0\,1\,0\,0\,1\,1\,1\,0\,0\,1\,0\,0\,0 &\qquad 1296 &\qquad 192 \\
    3 &\qquad    7 &\qquad 0\,0\,0\,0\,0\,1\,0\,2\,1\,0\,1\,0\,2\,0\,1\,0\,0\,0\,1\,1\,1\,0\,0\,0\,1\,0\,0 &\qquad  648 &\qquad 193 \\
    1 &\qquad  -15 &\qquad 0\,0\,0\,0\,0\,1\,0\,2\,1\,0\,1\,0\,2\,1\,0\,0\,0\,0\,1\,0\,2\,0\,0\,0\,1\,0\,0 &\qquad  432 &\qquad 194 \\
   -2 &\qquad    2 &\qquad 0\,0\,0\,0\,0\,1\,0\,2\,1\,0\,1\,0\,3\,0\,0\,0\,0\,0\,0\,1\,2\,0\,0\,0\,1\,0\,0 &\qquad  648 &\qquad 195 \\
    3 &\qquad    7 &\qquad 0\,0\,0\,0\,0\,1\,0\,2\,1\,0\,1\,1\,0\,0\,1\,1\,0\,0\,1\,1\,0\,2\,0\,0\,0\,0\,0 &\qquad  648 &\qquad 196 \\
   -3 &\qquad  -15 &\qquad 0\,0\,0\,0\,0\,1\,0\,2\,1\,0\,1\,1\,0\,0\,1\,1\,0\,0\,2\,0\,0\,1\,1\,0\,0\,0\,0 &\qquad 1296 &\qquad 197 \\
    3 &\qquad   47 &\qquad 0\,0\,0\,0\,0\,1\,0\,2\,1\,0\,1\,1\,0\,1\,0\,1\,0\,0\,2\,0\,0\,1\,0\,1\,0\,0\,0 &\qquad  648 &\qquad 198 \\
    1 &\qquad   53 &\qquad 0\,0\,0\,0\,0\,1\,0\,2\,1\,0\,1\,1\,1\,0\,0\,0\,0\,1\,2\,0\,0\,1\,1\,0\,0\,0\,0 &\qquad 1296 &\qquad 199 \\
    0 &\qquad  -24 &\qquad 0\,0\,0\,0\,0\,1\,0\,2\,1\,0\,1\,1\,1\,0\,0\,1\,0\,0\,0\,1\,1\,2\,0\,0\,0\,0\,0 &\qquad  648 &\qquad 200 \\
    2 &\qquad   26 &\qquad 0\,0\,0\,0\,0\,1\,0\,2\,1\,0\,1\,1\,1\,0\,0\,1\,0\,0\,1\,0\,1\,1\,1\,0\,0\,0\,0 &\qquad 1296 &\qquad 201 \\
   -2 &\qquad   38 &\qquad 0\,0\,0\,0\,0\,1\,0\,2\,1\,0\,1\,1\,1\,0\,0\,1\,0\,0\,1\,1\,0\,1\,0\,1\,0\,0\,0 &\qquad 1296 &\qquad 202 \\
    0 &\qquad  -64 &\qquad 0\,0\,0\,0\,0\,1\,0\,2\,1\,0\,1\,1\,1\,0\,0\,1\,0\,0\,2\,0\,0\,0\,1\,1\,0\,0\,0 &\qquad 1296 &\qquad 203 \\
   -2 &\qquad   -2 &\qquad 0\,0\,0\,0\,0\,1\,0\,2\,1\,0\,1\,1\,1\,0\,1\,0\,0\,0\,1\,1\,0\,1\,0\,0\,1\,0\,0 &\qquad 1296 &\qquad 204 \\
    3 &\qquad   47 &\qquad 0\,0\,0\,0\,0\,1\,0\,2\,1\,0\,1\,1\,1\,0\,1\,0\,0\,0\,2\,0\,0\,0\,1\,0\,1\,0\,0 &\qquad 1296 &\qquad 205 \\
    2 &\qquad  -14 &\qquad 0\,0\,0\,0\,0\,1\,0\,2\,1\,0\,1\,1\,1\,1\,0\,0\,0\,0\,1\,0\,1\,1\,0\,0\,1\,0\,0 &\qquad 1296 &\qquad 206 \\
   -3 &\qquad  -15 &\qquad 0\,0\,0\,0\,0\,1\,0\,2\,1\,0\,1\,1\,1\,1\,0\,0\,0\,0\,2\,0\,0\,0\,0\,1\,1\,0\,0 &\qquad 1296 &\qquad 207 \\
    1 &\qquad   53 &\qquad 0\,0\,0\,0\,0\,1\,0\,2\,1\,0\,1\,1\,1\,1\,0\,0\,0\,0\,2\,0\,0\,1\,0\,0\,0\,0\,1 &\qquad 1296 &\qquad 208 \\
   -1 &\qquad   59 &\qquad 0\,0\,0\,0\,0\,1\,0\,2\,1\,0\,1\,1\,2\,0\,0\,0\,0\,0\,1\,0\,1\,0\,1\,0\,1\,0\,0 &\qquad 1296 &\qquad 209 \\
    1 &\qquad  -27 &\qquad 0\,0\,0\,0\,0\,1\,0\,2\,1\,0\,1\,1\,2\,0\,0\,0\,0\,0\,1\,1\,0\,0\,0\,1\,1\,0\,0 &\qquad 1296 &\qquad 210 \\
   -2 &\qquad -106 &\qquad 0\,0\,0\,0\,0\,1\,0\,2\,1\,0\,1\,1\,2\,0\,0\,0\,0\,0\,2\,0\,0\,0\,1\,0\,0\,0\,1 &\qquad  648 &\qquad 211 \\
   -2 &\qquad  -58 &\qquad 0\,0\,0\,0\,0\,1\,0\,2\,1\,0\,1\,2\,1\,0\,0\,0\,0\,0\,1\,0\,0\,1\,1\,0\,1\,0\,0 &\qquad 1296 &\qquad 212 \\
    1 &\qquad    5 &\qquad 0\,0\,0\,0\,0\,1\,0\,2\,1\,0\,2\,0\,0\,0\,1\,1\,0\,0\,2\,0\,0\,1\,0\,1\,0\,0\,0 &\qquad  648 &\qquad 213 \\
   -1 &\qquad   11 &\qquad 0\,0\,0\,0\,0\,1\,0\,2\,1\,0\,2\,0\,1\,0\,0\,0\,0\,1\,2\,0\,0\,1\,0\,1\,0\,0\,0 &\qquad  648 &\qquad 214 \\
    4 &\qquad    4 &\qquad 0\,0\,0\,0\,0\,1\,0\,2\,1\,0\,2\,0\,1\,0\,0\,1\,0\,0\,1\,0\,1\,1\,0\,1\,0\,0\,0 &\qquad 1296 &\qquad 215 \\
   -2 &\qquad  -10 &\qquad 0\,0\,0\,0\,0\,1\,0\,2\,1\,0\,2\,0\,1\,0\,0\,1\,0\,0\,2\,0\,0\,0\,0\,2\,0\,0\,0 &\qquad  648 &\qquad 216 \\
  \bottomrule
  \end{array}
  \]
  \bigskip
  \caption{The basis invariants in degree 12: part 3}
  \label{degree12table3}
  \end{table}

  \begin{table} \tiny
  \[
  \begin{array}{ccccc}
  I_{12} &\qquad I'_{12} &\qquad
  \text{MINIMAL REPRESENTATIVE} &\qquad \text{ORBIT SIZE} &\qquad \#
  \\
  \midrule
    0 &\qquad  -24 &\qquad 0\,0\,0\,0\,0\,1\,0\,2\,1\,0\,2\,0\,1\,0\,1\,0\,0\,0\,1\,0\,1\,1\,0\,0\,1\,0\,0 &\qquad 1296 &\qquad 217 \\
    1 &\qquad    5 &\qquad 0\,0\,0\,0\,0\,1\,0\,2\,1\,0\,2\,0\,1\,0\,1\,0\,0\,0\,2\,0\,0\,0\,0\,1\,1\,0\,0 &\qquad 1296 &\qquad 218 \\
   -1 &\qquad   11 &\qquad 0\,0\,0\,0\,0\,1\,0\,2\,1\,0\,2\,0\,1\,0\,1\,0\,0\,0\,2\,0\,0\,1\,0\,0\,0\,0\,1 &\qquad 1296 &\qquad 219 \\
   -3 &\qquad   17 &\qquad 0\,0\,0\,0\,0\,1\,0\,2\,1\,0\,2\,0\,2\,0\,0\,0\,0\,0\,1\,0\,1\,0\,0\,1\,1\,0\,0 &\qquad 1296 &\qquad 220 \\
    2 &\qquad  -22 &\qquad 0\,0\,0\,0\,0\,1\,0\,2\,1\,0\,2\,0\,2\,0\,0\,0\,0\,0\,2\,0\,0\,0\,0\,1\,0\,0\,1 &\qquad  648 &\qquad 221 \\
    2 &\qquad   10 &\qquad 0\,0\,0\,0\,0\,1\,0\,2\,1\,0\,2\,1\,1\,0\,0\,0\,0\,0\,1\,0\,0\,1\,0\,1\,1\,0\,0 &\qquad 1296 &\qquad 222 \\
    0 &\qquad   32 &\qquad 0\,0\,0\,0\,0\,1\,0\,2\,1\,1\,0\,0\,0\,1\,1\,1\,0\,0\,1\,1\,1\,1\,0\,0\,0\,0\,0 &\qquad 1296 &\qquad 223 \\
   -2 &\qquad   22 &\qquad 0\,0\,0\,0\,0\,1\,0\,2\,1\,1\,0\,0\,1\,0\,1\,0\,0\,1\,1\,2\,0\,1\,0\,0\,0\,0\,0 &\qquad  324 &\qquad 224 \\
    0 &\qquad   -8 &\qquad 0\,0\,0\,0\,0\,1\,0\,2\,1\,1\,0\,0\,1\,0\,1\,1\,0\,0\,1\,1\,1\,0\,1\,0\,0\,0\,0 &\qquad 1296 &\qquad 225 \\
   -2 &\qquad  -18 &\qquad 0\,0\,0\,0\,0\,1\,0\,2\,1\,1\,0\,0\,1\,0\,1\,1\,0\,0\,1\,2\,0\,0\,0\,1\,0\,0\,0 &\qquad  648 &\qquad 226 \\
    2 &\qquad  106 &\qquad 0\,0\,0\,0\,0\,1\,0\,2\,1\,1\,0\,0\,1\,1\,0\,0\,0\,1\,1\,1\,1\,1\,0\,0\,0\,0\,0 &\qquad  648 &\qquad 227 \\
    0 &\qquad   -8 &\qquad 0\,0\,0\,0\,0\,1\,0\,2\,1\,1\,0\,0\,1\,1\,0\,1\,0\,0\,1\,1\,1\,0\,0\,1\,0\,0\,0 &\qquad 1296 &\qquad 228 \\
    0 &\qquad   32 &\qquad 0\,0\,0\,0\,0\,1\,0\,2\,1\,1\,0\,0\,1\,1\,1\,0\,0\,0\,1\,1\,1\,0\,0\,0\,1\,0\,0 &\qquad 1296 &\qquad 229 \\
    4 &\qquad  -20 &\qquad 0\,0\,0\,0\,0\,1\,0\,2\,1\,1\,0\,1\,0\,0\,1\,1\,0\,0\,1\,1\,0\,1\,1\,0\,0\,0\,0 &\qquad 1296 &\qquad 230 \\
    1 &\qquad   53 &\qquad 0\,0\,0\,0\,0\,1\,0\,2\,1\,1\,0\,1\,0\,1\,0\,0\,0\,1\,1\,1\,0\,2\,0\,0\,0\,0\,0 &\qquad  648 &\qquad 231 \\
    2 &\qquad  -54 &\qquad 0\,0\,0\,0\,0\,1\,0\,2\,1\,1\,0\,1\,0\,1\,0\,1\,0\,0\,1\,0\,1\,1\,1\,0\,0\,0\,0 &\qquad 1296 &\qquad 232 \\
   -2 &\qquad  -82 &\qquad 0\,0\,0\,0\,0\,1\,0\,2\,1\,1\,0\,1\,0\,1\,0\,1\,0\,0\,1\,1\,0\,1\,0\,1\,0\,0\,0 &\qquad 1296 &\qquad 233 \\
    0 &\qquad   32 &\qquad 0\,0\,0\,0\,0\,1\,0\,2\,1\,1\,0\,1\,0\,1\,1\,0\,0\,0\,1\,1\,0\,1\,0\,0\,1\,0\,0 &\qquad 1296 &\qquad 234 \\
   -2 &\qquad   22 &\qquad 0\,0\,0\,0\,0\,1\,0\,2\,1\,1\,0\,1\,0\,2\,0\,0\,0\,0\,1\,0\,1\,1\,0\,0\,1\,0\,0 &\qquad 1296 &\qquad 235 \\
    2 &\qquad  -54 &\qquad 0\,0\,0\,0\,0\,1\,0\,2\,1\,1\,0\,1\,1\,0\,0\,0\,0\,1\,1\,1\,0\,1\,1\,0\,0\,0\,0 &\qquad 1296 &\qquad 236 \\
    0 &\qquad   32 &\qquad 0\,0\,0\,0\,0\,1\,0\,2\,1\,1\,0\,1\,1\,0\,0\,1\,0\,0\,1\,1\,0\,0\,1\,1\,0\,0\,0 &\qquad 1296 &\qquad 237 \\
   -4 &\qquad  -36 &\qquad 0\,0\,0\,0\,0\,1\,0\,2\,1\,1\,0\,1\,1\,0\,1\,0\,0\,0\,1\,1\,0\,0\,1\,0\,1\,0\,0 &\qquad 1296 &\qquad 238 \\
    2 &\qquad   26 &\qquad 0\,0\,0\,0\,0\,1\,0\,2\,1\,1\,0\,1\,1\,1\,0\,0\,0\,0\,1\,1\,0\,0\,0\,1\,1\,0\,0 &\qquad 1296 &\qquad 239 \\
   -4 &\qquad -116 &\qquad 0\,0\,0\,0\,0\,1\,0\,2\,1\,1\,0\,1\,1\,1\,0\,0\,0\,0\,1\,1\,0\,1\,0\,0\,0\,0\,1 &\qquad  648 &\qquad 240 \\
   -2 &\qquad   22 &\qquad 0\,0\,0\,0\,0\,1\,0\,2\,1\,1\,1\,0\,0\,0\,1\,0\,0\,1\,1\,1\,0\,2\,0\,0\,0\,0\,0 &\qquad  648 &\qquad 241 \\
   -2 &\qquad   -2 &\qquad 0\,0\,0\,0\,0\,1\,0\,2\,1\,1\,1\,0\,0\,0\,1\,1\,0\,0\,1\,1\,0\,1\,0\,1\,0\,0\,0 &\qquad 1296 &\qquad 242 \\
    2 &\qquad   10 &\qquad 0\,0\,0\,0\,0\,1\,0\,2\,1\,1\,1\,0\,0\,0\,2\,0\,0\,0\,1\,1\,0\,1\,0\,0\,1\,0\,0 &\qquad  648 &\qquad 243 \\
    1 &\qquad   21 &\qquad 0\,0\,0\,0\,0\,1\,0\,3\,0\,0\,0\,1\,0\,1\,1\,1\,0\,0\,3\,0\,0\,0\,0\,1\,0\,0\,0 &\qquad  648 &\qquad 244 \\
    0 &\qquad   24 &\qquad 0\,0\,0\,0\,0\,1\,0\,3\,0\,0\,0\,1\,1\,0\,1\,1\,0\,0\,1\,1\,1\,1\,0\,0\,0\,0\,0 &\qquad  648 &\qquad 245 \\
   -1 &\qquad   -5 &\qquad 0\,0\,0\,0\,0\,1\,0\,3\,0\,0\,0\,1\,1\,0\,1\,1\,0\,0\,2\,1\,0\,0\,0\,1\,0\,0\,0 &\qquad 1296 &\qquad 246 \\
   -2 &\qquad  -10 &\qquad 0\,0\,0\,0\,0\,1\,0\,3\,0\,0\,0\,1\,1\,1\,0\,0\,0\,1\,3\,0\,0\,0\,0\,1\,0\,0\,0 &\qquad  108 &\qquad 247 \\
    2 &\qquad   10 &\qquad 0\,0\,0\,0\,0\,1\,0\,3\,0\,0\,0\,1\,2\,0\,0\,0\,0\,1\,2\,1\,0\,0\,0\,1\,0\,0\,0 &\qquad  648 &\qquad 248 \\
    3 &\qquad  -17 &\qquad 0\,0\,0\,0\,0\,1\,0\,3\,0\,0\,0\,2\,1\,0\,0\,1\,0\,0\,1\,1\,0\,1\,0\,1\,0\,0\,0 &\qquad  432 &\qquad 249 \\
   -2 &\qquad   22 &\qquad 0\,0\,0\,0\,0\,1\,0\,3\,0\,0\,0\,2\,1\,0\,0\,1\,0\,0\,2\,0\,0\,0\,1\,1\,0\,0\,0 &\qquad  648 &\qquad 250 \\
    2 &\qquad   42 &\qquad 0\,0\,0\,0\,0\,1\,0\,3\,0\,0\,0\,2\,2\,0\,0\,0\,0\,0\,2\,0\,0\,0\,1\,0\,0\,0\,1 &\qquad  216 &\qquad 251 \\
    1 &\qquad    1 &\qquad 0\,0\,0\,0\,0\,1\,0\,3\,0\,0\,0\,3\,0\,0\,0\,1\,0\,0\,0\,1\,0\,3\,0\,0\,0\,0\,0 &\qquad   72 &\qquad 252 \\
   -4 &\qquad   -4 &\qquad 0\,0\,0\,0\,0\,1\,0\,3\,0\,0\,1\,1\,1\,0\,0\,1\,0\,0\,1\,0\,1\,1\,0\,1\,0\,0\,0 &\qquad  648 &\qquad 253 \\
    2 &\qquad  -22 &\qquad 0\,0\,0\,0\,0\,1\,0\,3\,0\,1\,0\,0\,1\,0\,1\,0\,0\,1\,1\,1\,1\,1\,0\,0\,0\,0\,0 &\qquad  648 &\qquad 254 \\
    2 &\qquad   18 &\qquad 0\,0\,0\,0\,0\,1\,0\,3\,0\,1\,0\,0\,1\,0\,1\,1\,0\,0\,1\,1\,1\,0\,0\,1\,0\,0\,0 &\qquad  216 &\qquad 255 \\
    2 &\qquad   58 &\qquad 0\,0\,0\,0\,0\,1\,0\,3\,0\,1\,0\,1\,0\,1\,0\,1\,0\,0\,1\,0\,1\,1\,0\,1\,0\,0\,0 &\qquad  648 &\qquad 256 \\
   -4 &\qquad    4 &\qquad 0\,0\,0\,0\,0\,1\,1\,1\,1\,0\,0\,1\,0\,2\,0\,1\,0\,0\,1\,1\,1\,1\,0\,0\,0\,0\,0 &\qquad  324 &\qquad 257 \\
    3 &\qquad   47 &\qquad 0\,0\,0\,0\,0\,1\,1\,1\,1\,0\,0\,1\,0\,2\,0\,1\,0\,0\,2\,1\,0\,0\,0\,1\,0\,0\,0 &\qquad 1296 &\qquad 258 \\
    0 &\qquad   24 &\qquad 0\,0\,0\,0\,0\,1\,1\,1\,1\,0\,0\,1\,1\,1\,0\,0\,1\,0\,1\,1\,1\,1\,0\,0\,0\,0\,0 &\qquad  648 &\qquad 259 \\
   -2 &\qquad  -42 &\qquad 0\,0\,0\,0\,0\,1\,1\,1\,1\,0\,0\,1\,1\,1\,0\,0\,1\,0\,2\,1\,0\,0\,0\,1\,0\,0\,0 &\qquad 1296 &\qquad 260 \\
    4 &\qquad   20 &\qquad 0\,0\,0\,0\,0\,1\,1\,1\,1\,0\,0\,1\,1\,2\,0\,0\,0\,0\,1\,1\,1\,0\,0\,0\,1\,0\,0 &\qquad  648 &\qquad 261 \\
   -1 &\qquad  -21 &\qquad 0\,0\,0\,0\,0\,1\,1\,1\,1\,0\,0\,1\,1\,2\,0\,0\,0\,0\,2\,1\,0\,0\,0\,0\,0\,0\,1 &\qquad  648 &\qquad 262 \\
   -3 &\qquad   17 &\qquad 0\,0\,0\,0\,0\,1\,1\,1\,1\,0\,0\,2\,0\,1\,0\,1\,0\,0\,0\,2\,0\,2\,0\,0\,0\,0\,0 &\qquad  648 &\qquad 263 \\
    4 &\qquad  -20 &\qquad 0\,0\,0\,0\,0\,1\,1\,1\,1\,0\,0\,2\,0\,1\,0\,1\,0\,0\,1\,1\,0\,1\,1\,0\,0\,0\,0 &\qquad 1296 &\qquad 264 \\
   -1 &\qquad   11 &\qquad 0\,0\,0\,0\,0\,1\,1\,1\,1\,0\,0\,2\,0\,1\,0\,1\,0\,0\,2\,0\,0\,0\,2\,0\,0\,0\,0 &\qquad 1296 &\qquad 265 \\
    0 &\qquad   16 &\qquad 0\,0\,0\,0\,0\,1\,1\,1\,1\,0\,0\,2\,0\,2\,0\,0\,0\,0\,1\,1\,0\,1\,0\,0\,1\,0\,0 &\qquad  648 &\qquad 266 \\
    4 &\qquad    4 &\qquad 0\,0\,0\,0\,0\,1\,1\,1\,1\,0\,0\,2\,1\,1\,0\,0\,0\,0\,0\,2\,0\,1\,0\,0\,1\,0\,0 &\qquad 1296 &\qquad 267 \\
   -4 &\qquad  -36 &\qquad 0\,0\,0\,0\,0\,1\,1\,1\,1\,0\,0\,2\,1\,1\,0\,0\,0\,0\,1\,1\,0\,0\,1\,0\,1\,0\,0 &\qquad 1296 &\qquad 268 \\
    0 &\qquad   32 &\qquad 0\,0\,0\,0\,0\,1\,1\,1\,1\,0\,1\,0\,0\,1\,1\,1\,0\,0\,2\,0\,1\,0\,1\,0\,0\,0\,0 &\qquad 1296 &\qquad 269 \\
    2 &\qquad   18 &\qquad 0\,0\,0\,0\,0\,1\,1\,1\,1\,0\,1\,0\,1\,0\,1\,0\,1\,0\,1\,1\,1\,1\,0\,0\,0\,0\,0 &\qquad  648 &\qquad 270 \\
    2 &\qquad  -22 &\qquad 0\,0\,0\,0\,0\,1\,1\,1\,1\,0\,1\,0\,1\,0\,1\,1\,0\,0\,1\,1\,1\,0\,1\,0\,0\,0\,0 &\qquad  648 &\qquad 271 \\
   -2 &\qquad  -42 &\qquad 0\,0\,0\,0\,0\,1\,1\,1\,1\,0\,1\,0\,1\,1\,0\,0\,0\,1\,2\,0\,1\,0\,1\,0\,0\,0\,0 &\qquad  648 &\qquad 272 \\
   -4 &\qquad  -84 &\qquad 0\,0\,0\,0\,0\,1\,1\,1\,1\,0\,1\,0\,1\,1\,1\,0\,0\,0\,1\,1\,1\,0\,0\,0\,1\,0\,0 &\qquad  216 &\qquad 273 \\
   -2 &\qquad  -42 &\qquad 0\,0\,0\,0\,0\,1\,1\,1\,1\,0\,1\,0\,2\,0\,0\,0\,0\,1\,1\,1\,1\,0\,1\,0\,0\,0\,0 &\qquad  324 &\qquad 274 \\
   -2 &\qquad   -2 &\qquad 0\,0\,0\,0\,0\,1\,1\,1\,1\,0\,1\,0\,2\,0\,0\,0\,1\,0\,1\,1\,1\,0\,0\,1\,0\,0\,0 &\qquad  324 &\qquad 275 \\
   -4 &\qquad   36 &\qquad 0\,0\,0\,0\,0\,1\,1\,1\,1\,0\,1\,1\,0\,0\,1\,1\,0\,0\,1\,1\,0\,1\,1\,0\,0\,0\,0 &\qquad 1296 &\qquad 276 \\
    2 &\qquad   58 &\qquad 0\,0\,0\,0\,0\,1\,1\,1\,1\,0\,1\,1\,0\,1\,0\,1\,0\,0\,1\,0\,1\,1\,1\,0\,0\,0\,0 &\qquad  648 &\qquad 277 \\
   -2 &\qquad  -50 &\qquad 0\,0\,0\,0\,0\,1\,1\,1\,1\,0\,1\,1\,0\,1\,0\,1\,0\,0\,1\,1\,0\,1\,0\,1\,0\,0\,0 &\qquad 1296 &\qquad 278 \\
    2 &\qquad   18 &\qquad 0\,0\,0\,0\,0\,1\,1\,1\,1\,0\,1\,1\,0\,1\,1\,0\,0\,0\,1\,1\,0\,1\,0\,0\,1\,0\,0 &\qquad  648 &\qquad 279 \\
    0 &\qquad   32 &\qquad 0\,0\,0\,0\,0\,1\,1\,1\,1\,0\,1\,1\,0\,1\,1\,0\,0\,0\,2\,0\,0\,1\,0\,0\,0\,1\,0 &\qquad 1296 &\qquad 280 \\
   -2 &\qquad  -10 &\qquad 0\,0\,0\,0\,0\,1\,1\,1\,1\,0\,1\,1\,1\,0\,0\,0\,0\,1\,1\,1\,0\,1\,1\,0\,0\,0\,0 &\qquad 1296 &\qquad 281 \\
    0 &\qquad  -64 &\qquad 0\,0\,0\,0\,0\,1\,1\,1\,1\,0\,1\,1\,1\,0\,0\,0\,0\,1\,2\,0\,0\,0\,2\,0\,0\,0\,0 &\qquad 1296 &\qquad 282 \\
    0 &\qquad  -56 &\qquad 0\,0\,0\,0\,0\,1\,1\,1\,1\,0\,1\,1\,1\,0\,0\,0\,1\,0\,1\,1\,0\,1\,0\,1\,0\,0\,0 &\qquad 1296 &\qquad 283 \\
    2 &\qquad  106 &\qquad 0\,0\,0\,0\,0\,1\,1\,1\,1\,0\,1\,1\,1\,0\,0\,0\,1\,0\,2\,0\,0\,0\,1\,1\,0\,0\,0 &\qquad  648 &\qquad 284 \\
    0 &\qquad  -48 &\qquad 0\,0\,0\,0\,0\,1\,1\,1\,1\,0\,1\,1\,1\,0\,0\,1\,0\,0\,1\,0\,1\,0\,2\,0\,0\,0\,0 &\qquad 1296 &\qquad 285 \\
    2 &\qquad   58 &\qquad 0\,0\,0\,0\,0\,1\,1\,1\,1\,0\,1\,1\,1\,0\,0\,1\,0\,0\,1\,1\,0\,0\,1\,1\,0\,0\,0 &\qquad 1296 &\qquad 286 \\
    6 &\qquad   46 &\qquad 0\,0\,0\,0\,0\,1\,1\,1\,1\,0\,1\,1\,1\,0\,1\,0\,0\,0\,1\,1\,0\,1\,0\,0\,0\,1\,0 &\qquad  648 &\qquad 287 \\
   -4 &\qquad  -36 &\qquad 0\,0\,0\,0\,0\,1\,1\,1\,1\,0\,1\,1\,1\,0\,1\,0\,0\,0\,2\,0\,0\,0\,1\,0\,0\,1\,0 &\qquad 1296 &\qquad 288 \\
  \bottomrule
  \end{array}
  \]
  \bigskip
  \caption{The basis invariants in degree 12: part 4}
  \label{degree12table4}
  \end{table}

  \begin{table} \tiny
  \[
  \begin{array}{ccccc}
  I_{12} &\qquad I'_{12} &\qquad
  \text{MINIMAL REPRESENTATIVE} &\qquad \text{ORBIT SIZE} &\qquad \#
  \\
  \midrule
    0 &\qquad  -96 &\qquad 0\,0\,0\,0\,0\,1\,1\,1\,1\,0\,1\,1\,1\,1\,0\,0\,0\,0\,1\,1\,0\,1\,0\,0\,0\,0\,1 &\qquad  648 &\qquad 289 \\
    2 &\qquad  106 &\qquad 0\,0\,0\,0\,0\,1\,1\,1\,1\,0\,1\,1\,1\,1\,0\,0\,0\,0\,2\,0\,0\,0\,1\,0\,0\,0\,1 &\qquad 1296 &\qquad 290 \\
    4 &\qquad -108 &\qquad 0\,0\,0\,0\,0\,1\,1\,1\,1\,1\,1\,0\,0\,0\,1\,0\,0\,1\,1\,1\,0\,1\,1\,0\,0\,0\,0 &\qquad  324 &\qquad 291 \\
    2 &\qquad   10 &\qquad 0\,0\,0\,0\,0\,1\,1\,2\,0\,0\,0\,1\,1\,1\,0\,0\,0\,1\,2\,1\,0\,0\,0\,1\,0\,0\,0 &\qquad  324 &\qquad 292 \\
   -4 &\qquad  -20 &\qquad 0\,0\,0\,0\,0\,1\,1\,2\,0\,0\,0\,1\,2\,0\,0\,0\,0\,1\,1\,2\,0\,0\,0\,1\,0\,0\,0 &\qquad  324 &\qquad 293 \\
    1 &\qquad   21 &\qquad 0\,0\,0\,0\,0\,1\,1\,2\,0\,0\,0\,2\,0\,2\,0\,0\,0\,0\,2\,0\,0\,1\,0\,0\,0\,0\,1 &\qquad 1296 &\qquad 294 \\
   -1 &\qquad  -13 &\qquad 0\,0\,0\,0\,0\,1\,1\,2\,0\,0\,0\,2\,1\,0\,0\,0\,0\,1\,0\,2\,0\,2\,0\,0\,0\,0\,0 &\qquad  432 &\qquad 295 \\
    0 &\qquad   16 &\qquad 0\,0\,0\,0\,0\,1\,1\,2\,0\,0\,0\,2\,1\,0\,0\,0\,0\,1\,1\,1\,0\,1\,1\,0\,0\,0\,0 &\qquad 1296 &\qquad 296 \\
    4 &\qquad   84 &\qquad 0\,0\,0\,0\,0\,1\,1\,2\,0\,0\,0\,2\,1\,1\,0\,0\,0\,0\,1\,1\,0\,1\,0\,0\,0\,0\,1 &\qquad  648 &\qquad 297 \\
   -3 &\qquad  -63 &\qquad 0\,0\,0\,0\,0\,1\,1\,2\,0\,0\,0\,2\,1\,1\,0\,0\,0\,0\,2\,0\,0\,0\,1\,0\,0\,0\,1 &\qquad 1296 &\qquad 298 \\
    2 &\qquad   10 &\qquad 0\,0\,0\,0\,0\,1\,1\,2\,0\,0\,1\,0\,1\,0\,1\,0\,0\,1\,2\,0\,1\,0\,1\,0\,0\,0\,0 &\qquad  648 &\qquad 299 \\
   -2 &\qquad  -26 &\qquad 0\,0\,0\,0\,0\,1\,1\,2\,0\,0\,1\,0\,2\,0\,0\,0\,0\,1\,1\,0\,2\,0\,1\,0\,0\,0\,0 &\qquad  324 &\qquad 300 \\
    0 &\qquad  -48 &\qquad 0\,0\,0\,0\,0\,1\,1\,2\,0\,0\,1\,1\,1\,0\,0\,0\,0\,1\,1\,0\,1\,1\,1\,0\,0\,0\,0 &\qquad 1296 &\qquad 301 \\
    2 &\qquad   26 &\qquad 0\,0\,0\,0\,0\,1\,1\,2\,0\,0\,1\,1\,1\,0\,0\,0\,0\,1\,1\,1\,0\,1\,0\,1\,0\,0\,0 &\qquad 1296 &\qquad 302 \\
   -2 &\qquad   22 &\qquad 0\,0\,0\,0\,0\,1\,1\,2\,0\,0\,1\,1\,1\,0\,0\,0\,0\,1\,2\,0\,0\,0\,1\,1\,0\,0\,0 &\qquad  648 &\qquad 303 \\
   -4 &\qquad  -36 &\qquad 0\,0\,0\,0\,0\,1\,1\,2\,0\,0\,1\,1\,1\,0\,1\,0\,0\,0\,1\,1\,0\,1\,0\,0\,0\,0\,1 &\qquad 1296 &\qquad 304 \\
    1 &\qquad   53 &\qquad 0\,0\,0\,0\,0\,1\,1\,2\,0\,0\,1\,1\,1\,0\,1\,0\,0\,0\,2\,0\,0\,0\,1\,0\,0\,0\,1 &\qquad 1296 &\qquad 305 \\
    3 &\qquad  -33 &\qquad 0\,0\,0\,0\,0\,1\,1\,2\,0\,0\,1\,1\,1\,1\,0\,0\,0\,0\,2\,0\,0\,0\,0\,1\,0\,0\,1 &\qquad 1296 &\qquad 306 \\
   -2 &\qquad  -58 &\qquad 0\,0\,0\,0\,0\,1\,1\,2\,0\,0\,2\,0\,1\,0\,0\,0\,0\,1\,1\,0\,1\,1\,0\,1\,0\,0\,0 &\qquad 1296 &\qquad 307 \\
    4 &\qquad    4 &\qquad 0\,0\,0\,0\,0\,1\,1\,2\,0\,0\,2\,0\,1\,0\,1\,0\,0\,0\,1\,0\,1\,1\,0\,0\,0\,0\,1 &\qquad  648 &\qquad 308 \\
   -2 &\qquad   22 &\qquad 0\,0\,0\,0\,0\,1\,1\,2\,0\,0\,2\,1\,1\,0\,0\,0\,0\,0\,1\,0\,0\,1\,0\,1\,0\,0\,1 &\qquad  324 &\qquad 309 \\
    2 &\qquad   26 &\qquad 0\,0\,0\,0\,0\,1\,1\,2\,0\,1\,0\,1\,0\,1\,0\,0\,0\,1\,1\,0\,1\,1\,1\,0\,0\,0\,0 &\qquad 1296 &\qquad 310 \\
    0 &\qquad  -48 &\qquad 0\,0\,0\,0\,0\,1\,1\,2\,0\,1\,0\,1\,0\,1\,0\,0\,0\,1\,1\,1\,0\,1\,0\,1\,0\,0\,0 &\qquad  648 &\qquad 311 \\
   -4 &\qquad   44 &\qquad 0\,0\,0\,0\,0\,1\,1\,2\,0\,1\,0\,1\,1\,1\,0\,0\,0\,0\,1\,1\,0\,0\,0\,1\,0\,0\,1 &\qquad  648 &\qquad 312 \\
   -4 &\qquad   18 &\qquad 0\,0\,0\,0\,0\,2\,0\,2\,0\,0\,0\,2\,0\,0\,0\,2\,0\,0\,0\,2\,0\,2\,0\,0\,0\,0\,0 &\qquad   36 &\qquad 313 \\
    6 &\qquad  -34 &\qquad 0\,0\,0\,0\,0\,2\,0\,2\,0\,0\,0\,2\,0\,0\,0\,2\,0\,0\,1\,1\,0\,1\,1\,0\,0\,0\,0 &\qquad  324 &\qquad 314 \\
   -2 &\qquad   22 &\qquad 0\,0\,0\,0\,0\,2\,0\,2\,0\,0\,0\,2\,0\,0\,0\,2\,0\,0\,2\,0\,0\,0\,2\,0\,0\,0\,0 &\qquad  108 &\qquad 315 \\
   -5 &\qquad   23 &\qquad 0\,0\,0\,0\,0\,2\,0\,2\,0\,0\,0\,2\,1\,0\,0\,1\,0\,0\,1\,1\,0\,0\,1\,0\,1\,0\,0 &\qquad  648 &\qquad 316 \\
  -12 &\qquad   92 &\qquad 0\,0\,0\,0\,0\,2\,0\,2\,0\,0\,1\,1\,1\,0\,0\,1\,0\,0\,0\,1\,1\,1\,0\,0\,1\,0\,0 &\qquad  162 &\qquad 317 \\
    4 &\qquad  -76 &\qquad 0\,0\,0\,0\,0\,2\,0\,2\,0\,0\,1\,1\,1\,0\,0\,1\,0\,0\,1\,0\,1\,0\,1\,0\,1\,0\,0 &\qquad  648 &\qquad 318 \\
    0 &\qquad  128 &\qquad 0\,0\,0\,0\,0\,2\,0\,2\,0\,0\,1\,1\,1\,0\,0\,1\,0\,0\,2\,0\,0\,0\,1\,0\,0\,0\,1 &\qquad  324 &\qquad 319 \\
   -8 &\qquad   40 &\qquad 0\,0\,0\,0\,0\,2\,0\,2\,0\,1\,0\,1\,0\,0\,0\,1\,0\,1\,1\,1\,0\,1\,1\,0\,0\,0\,0 &\qquad  324 &\qquad 320 \\
   -4 &\qquad  108 &\qquad 0\,0\,0\,0\,0\,2\,0\,2\,0\,1\,0\,1\,0\,1\,0\,1\,0\,0\,1\,0\,1\,0\,1\,0\,1\,0\,0 &\qquad  324 &\qquad 321 \\
    4 &\qquad  164 &\qquad 0\,0\,0\,0\,0\,2\,0\,2\,0\,1\,0\,1\,0\,1\,0\,1\,0\,0\,1\,1\,0\,1\,0\,0\,0\,0\,1 &\qquad  324 &\qquad 322 \\
    2 &\qquad   10 &\qquad 0\,0\,0\,0\,0\,2\,0\,2\,0\,1\,0\,1\,1\,0\,0\,0\,0\,1\,1\,1\,0\,0\,1\,0\,1\,0\,0 &\qquad  648 &\qquad 323 \\
    8 &\qquad  -72 &\qquad 0\,0\,0\,0\,0\,2\,1\,1\,0\,0\,1\,1\,0\,0\,0\,1\,0\,1\,1\,1\,0\,1\,1\,0\,0\,0\,0 &\qquad 1296 &\qquad 324 \\
   -6 &\qquad   50 &\qquad 0\,0\,0\,0\,0\,2\,1\,1\,0\,0\,1\,1\,0\,1\,0\,1\,0\,0\,1\,0\,1\,1\,0\,0\,0\,1\,0 &\qquad  648 &\qquad 325 \\
   -2 &\qquad   38 &\qquad 0\,0\,0\,0\,0\,2\,1\,1\,0\,0\,1\,1\,0\,1\,0\,1\,0\,0\,1\,1\,0\,1\,0\,0\,0\,0\,1 &\qquad 1296 &\qquad 326 \\
    5 &\qquad   41 &\qquad 0\,0\,0\,0\,0\,2\,1\,1\,0\,0\,1\,1\,0\,1\,0\,1\,0\,0\,2\,0\,0\,0\,1\,0\,0\,0\,1 &\qquad  648 &\qquad 327 \\
    0 &\qquad  112 &\qquad 0\,0\,0\,0\,0\,2\,1\,1\,0\,0\,1\,1\,1\,0\,0\,0\,0\,1\,1\,1\,0\,0\,1\,0\,1\,0\,0 &\qquad 1296 &\qquad 328 \\
   -6 &\qquad  -30 &\qquad 0\,0\,0\,0\,0\,2\,1\,1\,0\,0\,1\,1\,1\,0\,0\,0\,0\,1\,1\,1\,0\,1\,0\,0\,0\,1\,0 &\qquad  432 &\qquad 329 \\
   -2 &\qquad  158 &\qquad 0\,0\,0\,0\,0\,2\,1\,1\,0\,0\,1\,1\,1\,0\,0\,0\,1\,0\,1\,0\,1\,0\,1\,0\,1\,0\,0 &\qquad  324 &\qquad 330 \\
   -1 &\qquad -181 &\qquad 0\,0\,0\,0\,0\,2\,1\,1\,0\,0\,1\,1\,1\,0\,0\,0\,1\,0\,2\,0\,0\,0\,1\,0\,0\,0\,1 &\qquad  648 &\qquad 331 \\
   -4 &\qquad -116 &\qquad 0\,0\,0\,0\,0\,2\,1\,1\,0\,0\,1\,1\,1\,0\,0\,1\,0\,0\,1\,1\,0\,0\,1\,0\,0\,0\,1 &\qquad 1296 &\qquad 332 \\
    2 &\qquad  -54 &\qquad 0\,0\,0\,0\,0\,2\,1\,1\,0\,0\,1\,1\,1\,1\,0\,0\,0\,0\,1\,1\,0\,0\,0\,0\,1\,0\,1 &\qquad  648 &\qquad 333 \\
   -6 &\qquad   34 &\qquad 0\,0\,0\,0\,0\,2\,1\,1\,0\,0\,1\,1\,1\,1\,0\,0\,0\,0\,2\,0\,0\,0\,0\,0\,0\,1\,1 &\qquad  324 &\qquad 334 \\
    4 &\qquad   20 &\qquad 0\,0\,0\,0\,0\,2\,1\,1\,0\,1\,1\,0\,0\,1\,0\,0\,0\,1\,1\,1\,0\,1\,0\,0\,0\,0\,1 &\qquad  324 &\qquad 335 \\
   24 &\qquad  -88 &\qquad 0\,0\,0\,0\,1\,1\,0\,1\,1\,0\,1\,1\,1\,0\,0\,1\,0\,0\,0\,1\,1\,1\,0\,0\,1\,0\,0 &\qquad   27 &\qquad 336 \\
   -8 &\qquad    8 &\qquad 0\,0\,0\,0\,1\,1\,0\,1\,1\,0\,1\,1\,1\,0\,0\,1\,0\,0\,1\,0\,1\,0\,1\,0\,1\,0\,0 &\qquad  648 &\qquad 337 \\
   10 &\qquad   -6 &\qquad 0\,0\,0\,0\,1\,1\,0\,1\,1\,1\,0\,1\,0\,0\,0\,1\,0\,1\,1\,1\,0\,1\,1\,0\,0\,0\,0 &\qquad  108 &\qquad 338 \\
   -8 &\qquad   48 &\qquad 0\,0\,0\,0\,1\,1\,0\,1\,1\,1\,0\,1\,0\,0\,0\,1\,1\,0\,1\,0\,1\,1\,1\,0\,0\,0\,0 &\qquad  324 &\qquad 339 \\
   -2 &\qquad   70 &\qquad 0\,0\,0\,0\,1\,1\,0\,1\,1\,1\,0\,1\,0\,0\,0\,1\,1\,0\,1\,1\,0\,1\,0\,1\,0\,0\,0 &\qquad  648 &\qquad 340 \\
    4 &\qquad  172 &\qquad 0\,0\,0\,0\,1\,1\,0\,1\,1\,1\,0\,1\,0\,1\,0\,1\,0\,0\,1\,0\,1\,1\,0\,0\,0\,1\,0 &\qquad  324 &\qquad 341 \\
   -2 &\qquad -130 &\qquad 0\,0\,0\,0\,1\,1\,0\,1\,1\,1\,0\,1\,0\,1\,0\,1\,0\,0\,1\,1\,0\,1\,0\,0\,0\,0\,1 &\qquad  324 &\qquad 342 \\
  -12 &\qquad   60 &\qquad 0\,0\,0\,0\,1\,1\,1\,0\,1\,0\,1\,1\,0\,0\,0\,1\,1\,0\,1\,0\,1\,1\,1\,0\,0\,0\,0 &\qquad  108 &\qquad 343 \\
   -4 &\qquad   76 &\qquad 0\,0\,0\,0\,1\,1\,1\,0\,1\,0\,1\,1\,0\,0\,0\,1\,1\,0\,1\,1\,0\,1\,0\,1\,0\,0\,0 &\qquad  324 &\qquad 344 \\
    2 &\qquad  -22 &\qquad 0\,0\,0\,0\,1\,1\,1\,0\,1\,0\,1\,1\,0\,0\,1\,1\,0\,0\,1\,1\,0\,1\,0\,0\,0\,1\,0 &\qquad 1296 &\qquad 345 \\
    4 &\qquad -108 &\qquad 0\,0\,0\,0\,1\,1\,1\,0\,1\,0\,1\,1\,0\,1\,0\,1\,0\,0\,1\,0\,1\,1\,0\,0\,0\,1\,0 &\qquad  648 &\qquad 346 \\
    0 &\qquad -136 &\qquad 0\,0\,0\,0\,1\,1\,1\,0\,1\,0\,1\,1\,1\,0\,0\,0\,0\,1\,1\,1\,0\,0\,1\,0\,1\,0\,0 &\qquad  216 &\qquad 347 \\
    6 &\qquad   46 &\qquad 0\,0\,0\,0\,1\,1\,1\,0\,1\,0\,1\,1\,1\,0\,0\,0\,1\,0\,1\,0\,1\,1\,0\,0\,0\,1\,0 &\qquad  648 &\qquad 348 \\
    4 &\qquad   92 &\qquad 0\,0\,0\,0\,1\,1\,1\,0\,1\,0\,1\,1\,1\,0\,0\,0\,1\,0\,1\,1\,0\,0\,0\,1\,1\,0\,0 &\qquad  648 &\qquad 349 \\
    0 &\qquad  -16 &\qquad 0\,0\,0\,0\,1\,1\,1\,0\,1\,0\,1\,1\,1\,0\,0\,0\,1\,0\,1\,1\,0\,1\,0\,0\,0\,0\,1 &\qquad 1296 &\qquad 350 \\
    8 &\qquad   88 &\qquad 0\,0\,1\,0\,0\,1\,1\,1\,0\,0\,0\,1\,0\,2\,0\,1\,0\,0\,1\,1\,0\,1\,0\,0\,0\,0\,1 &\qquad  216 &\qquad 351 \\
   -4 &\qquad -164 &\qquad 0\,0\,1\,0\,0\,1\,1\,1\,0\,0\,0\,1\,1\,1\,0\,0\,1\,0\,1\,1\,0\,1\,0\,0\,0\,0\,1 &\qquad  648 &\qquad 352 \\
    4 &\qquad  180 &\qquad 0\,0\,1\,0\,0\,1\,1\,1\,0\,0\,1\,0\,1\,0\,1\,0\,1\,0\,2\,0\,0\,0\,1\,0\,0\,0\,1 &\qquad  324 &\qquad 353 \\
    0 &\qquad  144 &\qquad 0\,0\,1\,0\,0\,1\,1\,1\,0\,0\,1\,0\,1\,0\,1\,1\,0\,0\,1\,1\,0\,0\,1\,0\,0\,0\,1 &\qquad  648 &\qquad 354 \\
   -4 &\qquad  236 &\qquad 0\,0\,1\,0\,0\,1\,1\,1\,0\,0\,1\,0\,1\,1\,0\,0\,0\,1\,1\,1\,0\,1\,0\,0\,0\,0\,1 &\qquad  108 &\qquad 355 \\
    0 &\qquad -128 &\qquad 0\,0\,1\,0\,0\,1\,1\,1\,0\,0\,1\,0\,1\,1\,0\,0\,0\,1\,2\,0\,0\,0\,1\,0\,0\,0\,1 &\qquad  216 &\qquad 356 \\
   -2 &\qquad  -82 &\qquad 0\,0\,1\,0\,1\,0\,1\,0\,1\,0\,1\,0\,1\,1\,0\,0\,0\,1\,2\,0\,0\,0\,0\,1\,0\,1\,0 &\qquad  324 &\qquad 357 \\
    6 &\qquad    6 &\qquad 0\,0\,1\,0\,1\,0\,1\,0\,1\,1\,0\,0\,0\,1\,1\,0\,1\,0\,1\,1\,0\,1\,0\,0\,0\,0\,1 &\qquad   72 &\qquad 358 \\
    2 &\qquad  222 &\qquad 0\,0\,1\,0\,1\,0\,2\,0\,0\,0\,2\,0\,1\,0\,0\,0\,0\,1\,1\,0\,0\,0\,0\,2\,0\,1\,0 &\qquad   36 &\qquad 359 \\
  \bottomrule
  \end{array}
  \]
  \bigskip
  \caption{The basis invariants in degree 12: part 5}
  \label{degree12table5}
  \end{table}



\end{document}